\documentclass[11pt]{article}

\usepackage{amssymb,amsmath,amscd,amsthm,xy,enumitem,mathrsfs,color,caption}
\xyoption{all}
\usepackage[pdftex]{graphicx}

\newtheorem{thm}{Theorem}[section]
\newtheorem{prp}[thm]{Proposition}
\newtheorem{lmm}[thm]{Lemma}

\newtheorem{crl}[thm]{Corollary}

\theoremstyle{definition}
\newtheorem{dfn}[thm]{Definition}
\newtheorem{eg}[thm]{Example}

\theoremstyle{remark}
\newtheorem{rmk}[thm]{Remark}

\numberwithin{equation}{section}

\def\lra{\longrightarrow}

\def\ri{\rightarrow}

\def\BE#1{\begin{equation}\label{#1}}
\def\EE{\end{equation}}
\def\lr#1{\langle#1\rangle}
\def\flr#1{\left\lfloor{#1}\right\rfloor}
\def\blr#1{\big\langle#1\big\rangle}
\def\ti#1{\tilde{#1}}
\def\wt#1{\widetilde{#1}}
\def\ov#1{\overline{#1}}
\def\eref#1{(\ref{#1})}
\def\tn#1{\textnormal{#1}}
\def\sf#1{\textsf{#1}}

\def\Ga{\Gamma}
\def\La{\Lambda}

\def\Si{\Sigma}

\def\de{\delta}

\def\gm{\gamma}

\def\na{\nabla}
\def\om{\omega}
\def\si{\sigma}
\def\th{\theta}

\def\vph{\varphi}

\def\ze{\zeta}

\def\fB{\mathfrak B}

\def\C{\mathbb C}
  
\def\cD{\mathcal{D}}

\def\ne{\textnormal{e}}

\def\cH{\mathcal H}

\def\bI{\mathbb I}
\def\fI{\mathfrak i}
\def\cJ{\mathcal J}
\def\fJ{\mathfrak j}

\def\cM{\mathcal M}

\def\fM{\mathfrak M}

\def\cO{\mathcal O}
\def\P{\mathbb P}
\def\R{\mathbb{R}}
\def\bT{\mathbb{T}}

\def\bS{\mathbb S}

\def\tU{\tn{U}}
\def\fU{\mathfrak U}

\def\Z{\mathbb{Z}}

\def\a{\mathbf a}
\def\fa{\mathfrak a}

\def\fc{\mathfrak c}
\def\ft{\mathfrak t}

\def\tnd{\textnormal{d}}

\def\ev{\tn{ev}}

\def\Hom{\tn{Hom}}
\def\id{\textnormal{id}}

\def\Id{\tn{Id}}

\def\pt{\tn{pt}}

\def\rk{\textnormal{rk}}

\def\tO{\tn{O}}
\def\SU{\tn{SU}}
\def\U{\tn{U}}

\def\top{\textnormal{top}}

\def\0{\mathbf 0}
\def\1{\mathbf 1}

\def\dbar{\bar\partial}
\def\prt{\partial}
\def\eset{\emptyset}
\def\i{\infty}
\def\bp{\bar\partial}
\def\w{\wedge}

\def\oD{\mathring{D}^2}

\begin{document}

\title{Orientability in Real Gromov-Witten Theory}
\author{Penka Georgieva and 
Aleksey Zinger\thanks{Partially supported by the IAS Fund for Math and NSF grants 
DMS 0635607 and 0846978}}
\date{\today}
\maketitle

\begin{abstract}
\noindent
The orientability problem in real Gromov-Witten theory is one
of the fundamental hurdles to enumerating real curves.
In this paper, we describe topological conditions on the target 
manifold which ensure that the uncompactified moduli spaces 
of real maps are orientable for {\it all} genera of and for {\it all} types of 
involutions on the domain.
In contrast to the typical approaches to this problem, we do not compute
the signs of any diffeomorphisms, but instead compare them.
Many projective complete intersections, including the renowned quintic threefold,
satisfy our topological conditions.
Our main result yields real Gromov-Witten invariants of arbitrary genus 
for real symplectic manifolds that satisfy these conditions and have empty real locus
and illustrates the significance of previously introduced moduli spaces of 
maps with crosscaps.
We also apply it to study the orientability of the moduli spaces
of real Hurwitz covers. 
\end{abstract}

\tableofcontents

\section{Introduction}
\label{intro_sec}

\noindent
The theory of $J$-holomorphic maps plays a prominent role in symplectic topology,
algebraic geometry, and string theory.
The foundational work of~\cite{Gr,Witten,McSa94,RT,FO,LT} has 
established the theory of (closed) Gromov-Witten invariants,
i.e.~counts of $J$-holomorphic maps from closed Riemann surfaces to symplectic manifolds.
In contrast, the theory of real Gromov-Witten invariants, 
i.e.~counts of $J$-holomorphic maps from symmetric Riemann surfaces commuting 
with the involutions on the domain and the target,
is still in early stages of development, especially in positive genera. 
The two main obstacles to defining real Gromov-Witten invariants are the potential 
non-orientability of the moduli space of real $J$-holomorphic maps and 
the existence of real codimension-one boundary strata.
In this paper, we address the former, obtaining sufficient topological  conditions 
on the target manifold for these moduli spaces to be orientable for {\it all} genera~of
and for {\it all} types of involutions on the domain; 
see Definition~\ref{realorient_dfn} and Theorem~\ref{thm_maps}.
Theorem~\ref{thm_maps}
yields real Gromov-Witten invariants of arbitrary genus 
for real symplectic manifolds that satisfy these conditions and have empty real locus;
see Theorem~\ref{thm_inv}.
Many projective complete intersections, including the quintic threefold
which plays a central role in Gromov-Witten theory,
satisfy our topological conditions; see Corollary~\ref{CIorient_crl}.\\

\noindent
The orientability question in real Gromov-Witten theory is studied in 
\cite{Wel, Pi, Sol, FOOO9, Ge2, Remi, Teh, GZ, GZ2}.
Real maps can be naturally divided into two groups,
depending on whether the involution~$\si$ on the domain~$\Si$
has separating fixed locus~$\Si^{\si}$ or~not.
In the first case, one can use bordered surfaces to obtain a good understanding 
of the orientability of the moduli spaces of such maps;
see \cite{Sol, FOOO9, Ge2, Remi, GZ, GZ2}.
In the second case, however, understanding the orientability in the bordered case 
is not sufficient beyond genus~1, due to the presence of real diffeomorphisms
of~$(\Si,\si)$ not preserving any half of~$\Si$; see Example~\ref{half_eg}.
The subtle effect of such diffeomorphisms on the orientability is hard to determine.
In~\cite{Remi}, this problem is studied for the diffeomorphisms of~$(\Si,\si)$
preserving some additional structure determined by a distinguished component 
of~$\Si^{\si}$ and a polarizing divisor on the target manifold~$X$,
obtaining orientability results for certain hypersurfaces in projective spaces.
In this paper, we adopt a fundamentally different approach: instead of computing
the signs of the actions induced by such diffeomorphisms, we directly compare
the orientation systems of the moduli spaces with the orientation systems
of certain bundles over~them naturally suggested by our previous study~\cite{GZ2};
see the end of this section for more details.
The orientable cases we discover include the orientable cases described in~\cite{Remi}
and go far beyond~them.\\

\noindent
An \textsf{involution} on a smooth manifold~$X$ is a diffeomorphism
$\phi\!:X\!\lra\!X$ such that $\phi\!\circ\!\phi\!=\!\id_X$.
Let
$$X^\phi=\big\{x\!\in\!X\!:~\phi(x)\!=\!x\big\}$$
denote the fixed locus.
An \sf{anti-symplectic involution~$\phi$} on a symplectic manifold $(X,\om)$
is an involution $\phi\!:X\!\lra\!X$ such that $\phi^*\om\!=\!-\om$.
A \sf{real symplectic manifold} is a triple $(X,\om,\phi)$ consisting 
of a symplectic manifold~$(X,\om)$ and an anti-symplectic involution~$\phi$.
For example, the~maps
\begin{alignat*}{2}
\tau_n\!:\P^{n-1}&\lra\P^{n-1}, &\qquad [Z_1,\ldots,Z_n]&\lra[\bar{Z}_1,\ldots,\bar{Z}_n],\\
\eta_{2m}\!: \P^{2m-1}& \lra\P^{2m-1},&\qquad
[Z_1,Z_2,\ldots,Z_{2m-1},Z_{2m}]&\lra 
\big[-\bar{Z}_2,\bar{Z}_1,\ldots,-\bar{Z}_{2m},\bar{Z}_{2m-1}\big],
\end{alignat*}
are anti-symplectic involutions with respect to the standard Fubini-Study symplectic
forms~$\om_n$ on~$\P^{n-1}$ and~$\om_{2m}$ on $\P^{2m-1}$, respectively.
If 
$$k\!\ge\!0, \qquad \a\equiv(a_1,\ldots,a_k)\in(\Z^+)^k\,,$$
and $X_{n;\a}\!\subset\!\P^{n-1}$ is a complete intersection of multi-degree~$\a$
preserved by~$\tau_n$,  $\tau_{n;\a}\!\equiv\!\tau_n|_{X_{n;\a}}$
is an anti-symplectic involution on $X_{n;\a}$ with respect to the symplectic form
$\om_{n;\a}\!=\!\om_n|_{X_{n;\a}}$. 
Similarly, if $X_{2m;\a}\!\subset\!\P^{2m-1}$ is preserved by~$\eta_{2m}$,
$\eta_{2m;\a}\!\equiv\!\eta_{2m}|_{X_{2m;\a}}$
is an anti-symplectic involution on $X_{2m;\a}$ with respect to the symplectic form
$\om_{2m;\a}\!=\!\om_{2m}|_{X_{2m;\a}}$.\\

\noindent
Let $(X,\phi)$ be a manifold with an involution.
A \sf{conjugation} on a complex vector bundle $V\!\lra\!X$ 
\sf{lifting} an involution~$\phi$ is a vector bundle homomorphism 
$\ti{\phi}\!:V\!\lra\!V$ covering~$\phi$ (or equivalently 
a vector bundle homomorphism  $\ti\phi\!:V\!\lra\!\phi^*V$ covering~$\id_X$)
such that the restriction of~$\ti\phi$ to each fiber is anti-complex linear
and $\ti{\phi}\!\circ\!\ti{\phi}\!=\!\id_V$.
A \sf{real bundle pair} $(V,\ti{\phi})\!\lra\!(X,\phi)$   
consists of a complex vector bundle $V\!\lra\!X$ and 
a conjugation~$\ti{\phi}$ on $V$ lifting~$\phi$.
For example, 
$$(TX,\tnd\phi)\lra(X,\phi) \qquad\hbox{and}\qquad
(X\!\times\!\C,\phi\!\times\!\fc_{\C})\lra(X,\phi),$$
where $\fc_{\C}\!:\C\!\lra\!\C$ is the standard conjugation on~$\C$,
are real bundle pairs.
For any real bundle pair $(V,\ti{\phi})\!\lra\!(X,\phi)$, 
we denote~by
$$\La_{\C}^{\top}(V,\ti{\phi})=(\La_{\C}^{\top}V,\La_{\C}^{\top}\ti{\phi})$$
the top exterior power of $V$ over $\C$ with the induced conjugation.
Direct sums, duals, and tensor products over~$\C$ of real bundle pairs over~$(X,\phi)$
are again real bundle pairs over~$(X,\phi)$.\\

\noindent
A \sf{symmetric surface} $(\Si,\si)$ is a closed connected oriented smooth 
surface~$\Si$ (manifold of real dimension~2) with an orientation-reversing involution~$\si$.
The fixed locus of~$\si$ is a disjoint union of circles.
If in addition $(X,\phi)$ is a manifold with an involution, 
a \sf{real map} 
$$u\!:(\Si,\si)\lra(X,\phi)$$ 
is a smooth map $u\!:\Si\!\lra\!X$ such that $u\!\circ\!\si=\phi\!\circ\!u$.
We denote the space of such maps by~$\fB_g(X)^{\phi,\si}$.\\

\noindent
For a symplectic manifold $(X,\om)$, we denote~by $\cJ_{\om}$
the space of $\om$-compatible almost complex structures on~$X$.
If $\phi$ is  an anti-symplectic involution on~$(X,\om)$, let 
$$\cJ_{\phi}=\big\{J\!\in\!\cJ_{\om}\!:\,\phi^*J\!=\!-J\big\}.$$
For a genus~$g$ symmetric surface~$(\Si,\si)$, we similarly denote by $\cJ_{\si}$
the space of complex structures on~$\Si$ compatible with the orientation such that 
$\si^*\fJ\!=\!-\fJ$.
For $J\!\in\!\cJ_{\phi}$, $\fJ\!\in\!\cJ_{\si}$, and
$u\!\in\!\fB_g(X)^{\phi,\si}$, let 
$$\dbar_{J,\fJ}u=\frac{1}{2}\big(\tnd u+J\circ\tnd u\!\circ\!\fJ\big).$$
If $l\!\in\!\Z^{\ge0}$, $J\!\in\!\cJ_{\phi}$, and $B\!\in\!H_2(X;\Z)$, let
\begin{equation*}\begin{split}
\fM_{g,l}(X,B;J)^{\phi,\si}= 
\big\{(u,(z_1,\si(z_1)),\ldots,(z_l,\si(z_l)),\fJ)\!\in\!
\fB_g(X)^{\phi,\si}\!\times\!\Si^{2l}\!\times\!\cJ_{\si}\!:\qquad&\\
~u_*[\Si]_{\Z}\!=\!B,~\dbar_{J,\fJ}u\!=\!0\big\}\big/\!\!\sim&
\end{split}\end{equation*}
be the moduli space of equivalence classes of degree~$B$ real $J$-holomorphic maps
from $(\Si,\si)$ to~$(X,\phi)$ with $l$ pairs of non-real conjugate distinct points; 
two $J$-holomorphic maps are equivalent in this space 
if they differ by an orientation-preserving diffeomorphism of~$\Si$ commuting with~$\si$.
Let 
$$\fM_{g,l}(X,B;J)^{\phi}=\bigcup_{\si}\fM_{g,l}(X,B;J)^{\phi,\si}$$
denote the union over all topological types of orientation-reversing involutions
on a genus~$g$ surface~$\Si$.
Using the geometric perturbations of~\cite{RT} adapted to the real case as
in \cite[Section~2]{Ge2}, we can perturb $\dbar_{J,\fJ}$ to achieve transversality; 
thus, we may assume that $\fM_{g,l}(X,B;J)^{\phi,\si}$ is an orbifold 
if the domain is stable. Under the assumptions of Remark~\ref{mfld_rem}, 
$\fM_{g,l}(X,B;J)^{\phi,\si}$ is a manifold outside codimension 2 strata and 
has a first Stiefel-Whitney class.  
 
\begin{dfn}\label{realorient_dfn}
A real symplectic manifold $(X,\om,\phi)$ is \sf{real-orientable} if
there exists a rank~1 real bundle pair $(L,\ti\phi)\!\lra\!(X,\phi)$ such~that 
\BE{realorient_e}w_2(TX^{\phi})=w_1(L^{\ti\phi})^2
\qquad\hbox{and}\qquad
\La_{\C}^{\top}(TX,\tnd\phi)=(L,\ti\phi)^{\otimes 2}\,.\EE
\end{dfn}

\begin{thm}\label{thm_maps} 
Let $(X,\om,\phi)$ be a real-orientable $2n$-manifold, $B\!\in\!H_2(X,\Z)$, 
$J\!\in\!\cJ_{\phi}$, $l\!\in\!\Z^{\ge0}$, and $(\Si,\si)$ be a symmetric surface
of genus $g\!\ge\!2$. 
\begin{enumerate}[label=(\arabic*),leftmargin=*]
\item\label{it_maps1} If $n$ is odd, then the moduli space $\fM_{g,l}(X,B;J)^{\phi,\si}$ is orientable.
\item If $g\!+\!2l\ge 4$, then
$$w_1(\fM_{g,l}(X,B;J)^{\phi,\si})=(n\!+\!1)\,\mathfrak{f}^*w_1(\cM_{g,l}^{\si}),$$
where $\mathfrak{f}\!:\fM_{g,l}(X,B;J)^{\phi,\si}\!\lra\!\cM_{g,l}^{\si}$ is the forgetful
morphism to the Deligne-Mumford moduli space of $\si$-compatible complex structures on~$\Si$.
\end{enumerate}
\end{thm}

\noindent
The genus~0 and~1 analogues of Theorem~\ref{thm_maps} are essentially 
\cite[Theorems~1.1,~1.2]{GZ}, respectively;
less general versions of \cite[Theorem~1.1]{GZ} are contained 
in \cite[Theorem~1.1]{FOOO9} and \cite[Theorem~1.3]{Teh}.
The second requirement in~\eref{realorient_e}  is not necessary
for the conclusion of Theorem~\ref{thm_maps} if $\Si\!-\!\Si^{\si}$
is disconnected; see \cite[Theorem~1.4]{GZ2}.
We note that Definition~\ref{realorient_dfn} forces the fixed locus $X^{\phi}$ 
to be orientable; so Theorem~\ref{thm_maps} does not consider any situations with
unorientable Lagrangians.\\

\noindent
Under the assumptions on $(X,\om,\phi)$ in Theorem~\ref{thm_maps}\ref{it_maps1},
an orientation on $\fM_{g,l}(X,B;J)^{\phi,\si}$ can be specified as follows.
Fix $\fJ\!\in\!\cJ_{\si}$ and choose an orientation of the tangent 
space of the Deligne-Mumford moduli space~$\cM_{g,0}^{\si}$ at~$\fJ$.
Let $[\Si,X]_B^{\phi,\si}$ denote the set of homotopy classes of real maps 
from $(\Si,\si)$ to~$(X,\phi)$ so that $u_*[\Si]\!=\!B$.
The group~$\cD^{\si}$ of orientation-preserving diffeomorphisms of~$\Si$ commuting with~$\si$
acts on $[\Si,X]_B^{\phi,\si}$ by composition on the right.
For each coset of this group action, 
choose a representative~$u_i$ and an orientation of 
the index of a linearization of a real Cauchy-Riemann operator in~$u_i^*TX$
compatible with~$\fJ$.
The latter can be obtained by choosing a spin structure on the real vector bundle
$$TX^{\phi}\oplus L^{\ti\phi}\oplus L^{\ti\phi}\lra X^{\phi}\,,$$
fixing the second identification in~\eref{realorient_e}, 
and choosing a symmetric half-surface~$\Si^b$, or a surface with crosscaps, 
doubling to~$\Si$;
see Section~\ref{prelim_sec} and \cite[Section~4]{GZ}.
The resulting orientation on $\fM_{g,l}(X,B;J)^{\phi,\si}$ does not depend
on~$J$.
However, if $g\!\ge\!2$, it does depend on
the choice of~$\Si^b$ for most topological types of orientation-reversing
involutions~$\si$ on~$\Si$.
This dependence is suggested by Example~\ref{half_eg} and illustrates
the significance of maps 
with crosscaps in real Gromov-Witten theory in positive genera;
the mathematical construction of moduli spaces of such maps in~\cite{GZ}
is motivated by their role in 
the description of localization data for real GW-invariants
of the quintic threefold in~\cite{Wal}.
%The precise dependence of the orientation of $\fM_{g,l}(X,B;J)^{\phi,\si}$
%on the choice of~$u_i$ and~$\Si^b$ can be determined from \cite{Remi}.

\begin{thm}\label{thm_inv} 
Let $(X,\om,\phi)$ be a real symplectic $2n$-manifold,
$g,l\!\in\!\Z^{\ge0}$, $B\!\in\!H_2(X;\Z)$, and $J\!\in\!\cJ_{\phi}$.
If $n\!\not\in\!2\Z$, $X^{\phi}\!=\!\eset$, and 
$$\La_{\C}^{\top}(TX,\tnd\phi)=(L,\ti\phi)^{\otimes 2}$$
for some rank~1 real bundle pair $(L,\ti\phi)\!\lra\!(X,\phi)$,
then the moduli space $\fM_{g,l}(X,B;J)^{\phi,\si}$ carries a virtual fundamental class
and thus gives rise to real genus~$g$ Gromov-Witten invariants of $(X,\om,\phi)$.
\end{thm}

\noindent
If $X^{\phi}\!=\!\eset$, $\fM_{g,l}(X,B;J)^{\phi,\si}\!=\!\eset$
for every topological type of orientation-reversing involutions~$\si$
on a genus~$g$ surface~$\Si$ except for the involutions~$\si_g$
with $\Si^{\si_g}\!=\!\eset$.
Thus,
$$\fM_{g,l}(X,B;J)^{\phi}=\fM_{g,l}(X,B;J)^{\phi,\si_g}\,.$$
By Theorem~\ref{thm_maps} and \cite[Theorems~1.1,~1.2]{GZ},
this moduli space is orientable under the assumptions of Theorem~\ref{thm_inv}.
The (virtual) codimension-one boundary of $\fM_{g,l}(X,B;J)^{\phi}$
consists of maps to~$X$ from two-component domains sending
the node to~$X^{\phi}$. 
If $X^{\phi}\!=\!\eset$, the boundary is empty and
the moduli space $\fM_{g,l}(X,B;J)^{\phi}$ with a choice of orientation
determines a virtual fundamental class,
obtained by a suitable adaptation of the usual VFC constructions of \cite{FO,LT},
as in \cite[Section~7]{Sol} and \cite[Remark~3.3]{Ge2}.

\begin{crl}\label{CIorient_crl}
Let $n\!\in\!\Z^+$, $g,k,l\!\in\!\Z^{\ge0}$, $\a\!\equiv\!(a_1,\ldots,a_k)\!\in\!(\Z^+)^k$, and 
$(\Si,\si)$ be a genus~$g$ symmetric surface.
\begin{enumerate}[label=(\arabic*),leftmargin=*]
\item If $n\!-\!k\in2\Z$,
$X_{n;\a}\!\subset\!\P^{n-1}$ is a complete intersection of multi-degree~$\a$
preserved by~$\tau_n$,
$$\sum_{i=1}^ka_i\equiv n \mod2, \qquad\hbox{and}\qquad
\sum_{i=1}^ka_i^2\equiv \sum_{i=1}^ka_i \mod4,$$
the moduli space  $\fM_{g,l}(X_{n;\a},B;J)^{\tau_{n;\a},\si}$
is orientable for every $B\!\in\!H_2(X_{n;\a};\Z)$ and $J\!\in\!\cJ_{\tau_{n;\a}}$.
\item If $k\!\in\!2\Z$, $X_{2n;\a}\!\subset\!\P^{2n-1}$ is a complete intersection of 
multi-degree~$\a$ preserved by~$\eta_{2n}$ and
$$a_1\!+\!\ldots\!+\!a_k\equiv 2n \mod4,$$
the moduli space $\fM_{g,l}(X_{2n;\a},B;J)^{\eta_{2n;\a},\si}$
carries a virtual fundamental class for every\linebreak
 \hbox{$B\!\in\!H_2(X_{2n;\a};\Z)$} 
and \hbox{$J\!\in\!\cJ_{\eta_{2n;\a}}$}
and thus gives rise to real genus~$g$ Gromov-Witten 
invariants of~$(X_{2n;\a},\om_{2n;\a},\eta_{2n;\a})$.
\end{enumerate}
\end{crl}

\noindent
In the genus~0 case, the conclusions of Corollary~\ref{CIorient_crl}
apply without the parity assumptions on~$k$ (which correspond to 
the dimension of~$X_{n;\a}$ being odd).\\

\noindent
Let $X_{n;\de}\!\subset\!\P^{n-1}$ denote a hypersurface of degree $\de\!\in\!\Z^+$ 
preserved by~$\tau_n$.
Theorem~\ref{thm_maps} applies to~$X_{n;\de}$ if
$$\de=0,1\mod\,4\qquad\hbox{and}\qquad \de\equiv n\mod2.$$
With the second condition strengthened to $\de\!\equiv\!n\mod4$,
the conclusion of Theorem~\ref{thm_maps} is obtained in \cite[Corollaire 2.4]{Remi}
under the additional assumption that $\Si^{\si}$ is a single circle;
if $\Si^{\si}$ consists of more than one circle, 
\cite[Corollaire~2.4]{Remi} shows that the conclusion of  Theorem~\ref{thm_maps}
holds after pulling back to a cover of $\fM_g(X,B;J)^{\phi,\si}$.\\

\noindent
By \cite[Lemma~3.1]{Teh}, the canonical line bundle~$K_X$ of
a Kahler Calabi-Yau manifold~$X$, i.e.~$K_X$ is trivial as a holomorphic line bundle, 
is trivial as a rank~1 real bundle pair with respect to any involution~$\phi$ 
which is anti-holomorphic with respect to the Kahler complex structure.
As noted in \cite[Section~2.1]{Teh},
the canonical line bundle~$K_X$ of a simply connected symplectic Calabi-Yau manifold~$(X,\om)$,
i.e.~$c_1(X,\om)\!=\!0$, is trivial with respect to any anti-symplectic involution~$\phi$
on~$X$.
Thus, Theorems~\ref{thm_maps} and~\ref{thm_inv} also apply to Kahler Calabi-Yau manifolds~$X$
with an  anti-holomorphic involution and to 
simply connected symplectic Calabi-Yau manifolds with any anti-symplectic involution,
provided the fixed locus is spin, i.e.~$w_2(X^{\phi})\!=\!0$, in the case of 
Theorem~\ref{thm_maps} and empty in the case of Theorem~\ref{thm_inv}.\\

\noindent
Theorem~\ref{thm_maps} and Corollary~\ref{CIorient_crl} can be extended to 
the moduli spaces $\fM_{g,k,l}(X,B;J)^{\phi,\si}$ of 
real maps with $k$~boundary and $l$~interior marked points,
as the effect of adding marked points on the sign of the relevant automorphisms
can be easily determined.
The moduli spaces $\fM_{g,k,l}(X,B;J)^{\phi,\si}$ typically have codimension-one boundary
and often of more than one type.
The codimension-one boundary stratum consisting of maps from $ \Si$ with
a bubble attached at a real point of the domain can be eliminated by the gluing 
procedure of \cite{Cho,Sol}, which is adapted to maps with decorated 
marked points in \cite[Section~3]{Ge2}.
By \cite[Theorems~1.3]{Ge2}, the proof of \cite[Corollary~6.1]{Ge2}, 
and \cite[Propositions~4.1,~4.2]{GZ2},
Theorem~\ref{thm_maps} can be extended to
the glued moduli space $\wt\fM_{g,0,l}(X,B;J)^{\phi,\si}$.
The remaining types of codimension-one boundary strata of $\fM_{g,0,l}(X,B;J)^{\phi,\si}$
correspond to one-nodal degenerations of~$\Si$ passing between
involutions on~$\Si$ of different topological types, 
as described in detail in  \cite[Section~4]{Klein}, 
\cite[Section~5]{Se91}, and~\cite[Sections~3,4]{Melissa}.
As suggested in \cite[Section~1.5]{PSW} and carried out in \cite[Section~3]{Teh} 
in the case $\Si\!=\!\P^1$, the moduli spaces $\ov\fM_{g,0,l}(X,B;J)^{\phi,\si}$
with different types of involutions~$\si$ on~$ \Si$ should in general be combined
to get well-defined invariants by gluing along codimension-one boundaries.
We intend to consider this question in a future paper.\\

\noindent
In Section~\ref{HurCov_sec}, we apply Theorem~\ref{thm_maps} to show
that some moduli spaces of real Hurwitz covers are orientable.
We also establish the orientability of such moduli spaces 
directly in some cases with the target $(X,\phi)\!=\!(\P^1,\eta)$;
Theorem~\ref{thm_maps} does not apply to these cases as
the rank~1 real bundle pair $(T\P^1,\tnd\eta)$ does not admit a real
square root.
This suggests that perhaps the second requirement in~\eref{realorient_e}
can be replaced by the requirement that the equivariant~$w_2$ of 
$\La_{\C}^{\top}(TX,\tnd\phi)$ be a square class;
by \cite[Corollary~3.2]{GZ2}, the latter is 
the case if the second condition in~\eref{realorient_e} holds
or if $\pi_1(X)\!=\!0$ and $w_2(TX)\!=\!0$.
By \cite[Corollary~1.6]{GZ}, the $w_2$ requirement suffices whenever 
the domain of the maps is~$\P^1$.
On the other hand, the orientability problem for the 
moduli spaces of real Hurwitz covers appears to 
be a purely combinatorial question about the topology of various Hurwitz covers
and can perhaps be addressed by more classical methods.\\

\noindent
The typical approaches to the orientability problem in real Gromov-Witten
theory involve computing the signs of the actions of appropriate diffeomorphisms
on determinant lines of real Cauchy-Riemann operators over some coverings of
$\fM_g(X,B;J)^{\phi,\si}$, as done in \cite{Sol, FOOO9, Remi, GZ2}.
These approaches work as long as the relevant diffeomorphisms are homotopically fairly simple
and in particular preserve a bordered surface in~$\Si$ that doubles to~$\Si$
or map it to its conjugate half;
more general diffeomorphisms are considered in~\cite{Remi}.
In contrast to~\cite{Sol, FOOO9}, in~\cite{GZ2} we allowed the complex structure
on a bordered domain to vary and considered diffeomorphisms that interchange
the boundary components.
We discovered~that 
\begin{enumerate}[label=(\arabic*),leftmargin=*]
\item these diffeomorphisms act with the same signs~on a natural cover of $\cM_g^{\si}$ 
and on the determinant line bundle for the trivial rank~1 real bundle pair over~it;
\item the signs for the square of a rank~1 real bundle pair are often the same as
for the rank~1 trivial real bundle pair; 
\item the signs for a real bundle pair and its top exterior power are often 
related just by the parity of the rank of the bundle pair;
\end{enumerate}
see Corollary~2.2 and Propositions~4.1 and~4.2 in~\cite{GZ2}.
In this paper, we show that suitable interpretations of these statements
apply to arbitrary real diffeomorphisms of the closed surface.
Proposition~\ref{prp_domains}, which appears to be of its own interest, 
establishes~(1) in the general case.
Proposition~\ref{bndlsplit_prp} captures the phenomena~(2) and~(3)
for arbitrary diffeomorphisms and is used in the proof of 
Proposition~\ref{prp_domains}.
In contrast to the typical approaches, we compare the signs directly,
instead of computing each sign separately.
In Section~\ref{mainpf_sec}, we combine Proposition~\ref{bndlsplit_prp}
and Proposition~\ref{prp_domains} in order to confirm Theorem~\ref{thm_maps}
and then deduce Corollary~\ref{CIorient_crl}.

\begin{rmk}\label{Remi_rmk}
After completing this paper, we discovered that \cite[Proposition~1.2]{Remi}
contains what can be viewed as a more general version of one of the main stepping
stones in this paper, Proposition~\ref{bndlsplit_prp0}.
The former applies in broader range of cases; while the conclusion
of our Proposition~\ref{bndlsplit_prp0} need not hold in some of these cases,
the conclusion of \cite[Proposition~1.2]{Remi} suffices for the purposes 
of~\cite{Remi} and of this paper.
In the appendix, we provide an alternative, less technical, argument for this broader
result and extend our Theorem~\ref{thm_maps} further; 
unlike the main part of this paper, this argument relies on an actual sign computation.
As explained in the appendix, the benefit of this extension is unclear to us
at this point.
\end{rmk}

\noindent
We would like to thank E.~Brugell\'e, R.~Cretois, E.~Ionel, S.~Lisi,
M.~Liu, J.~Solomon, M.~Tehrani, G.~Tian, and J.~Welschinger
for related discussions.
The second author is also grateful to the IAS School of Mathematics for its hospitality 
during the period when the results in this paper were obtained.

\section{Setup}
\label{prelim_sec}

%\subsection{symmetric and sh surfaces}

\noindent
Let $(\Si,\si)$ be a genus~$g$ symmetric surface.
We denote by $|\si|_0\!\in\!\Z^{\ge0}$ the number of connected components of~$\Si^{\si}$;
each of them is a circle.
Let $\lr\si\!=\!0$ if the quotient $\Si/\si$ is orientable, 
i.e.~$ \Si\!-\! \Si^{\si}$ is disconnected, and $\lr\si\!=\!1$ otherwise. 
There are $\flr{\frac{3g+4}{2}}$ different topological types of orientation-reversing 
involutions $\si$ on $\Si$ classified by the triples $(g,|\si|_0,\lr\si)$; 
see \cite[Colollary~1.1]{Nat}. 
The two equivalence classes of orientation-reversing involutions on $S^2\!=\!\P^1$
are described in Example~\ref{HurCov_eg2}, while the three 
equivalence classes of orientation-reversing involutions on
$\bT\!=\!S^1\!\times\!S^1$ are described in~(2) of the proof of Theorem~\ref{HurCov_thm}.\\

\noindent
An \sf{oriented symmetric half-surface} (or simply \sf{oriented sh-surface}) 
is a pair $(\Si^b,c)$ consisting of an oriented bordered smooth surface~$\Si^b$ 
and an involution $c\!:\prt\Si^b\!\lra\!\prt\Si^b$ preserving each component
and the orientation of~$\prt\Si^b$.
The restriction of~$c$  to a boundary component 
is either the identity or the antipodal map
\BE{antip_e}\fa\!:S^1\lra S^1, \qquad z\lra -z,\EE
for a suitable identification $(\prt\Si^b)_i$ with $S^1\!\subset\!\C$;
the latter type of boundary structure is called \sf{crosscap} 
in the string theory literature.
We define
$$c_i=c|_{(\prt\Si^b)_i}, \qquad 
|c_i|= \begin{cases} 0,&\hbox{if}~c_i=\id;\\ 1,&\hbox{otherwise};\end{cases}
\qquad
|c|_k=\big|\{(\prt\Si^b)_i\!\subset\!\Si^b\!:\,|c_i|\!=\!k\}\big|\quad k=0,1.$$
Thus, $|c|_0$ is the number of standard boundary components of $(\Si^b,\prt\Si^b)$  
and $|c|_1$ is the number of crosscaps.
%We order the boundary components~$(\prt\Si^b)_i$ of~$\Si$ so that $|c_i|\!=\!0$
%for $i\!=\!1,\ldots,|c|_0$.
Up to isomorphism, each oriented sh-surface $(\Si^b,c)$ is determined by the genus~$g$ of~$\Si^b$,
the number~$|c|_0$ of ordinary boundary components, and the number~$|c|_1$ of crosscaps.
We denote by $(\Si_{g,m_0,m_1},c_{g,m_0,m_1})$ the genus~$g$ oriented sh-surface
with $|c_{g,m_0,m_1}|_0\!=\!m_0$ and $|c_{g,m_0,m_1}|_1\!=\!m_1$.\\

\noindent
An  oriented sh-surface $(\Si^b,c)$ of type $(g,m_0,m_1)$ \sf{doubles} 
to a symmetric surface~$(\Si, \si)$ of type 
$$(g(\Si), |\si|_0,\lr\si)=\begin{cases} (2g\!+\!m_0\!+\!m_1\!-\!1,m_0,0),& \text{if}~ m_1=0;\\
(2g\!+\!m_0\!+\!m_1\!-\!1,m_0,1),& \text{if}~m_1\neq 0;
\end{cases}$$
so that $\si$ restricts to~$c$ on the cutting circles (the boundary of~$\Si^b$);
see \cite[(1.6)]{GZ2}.
\begin{figure}[htb]
\begin{center}
\leavevmode
\includegraphics[width=0.6\textwidth,height=150px]{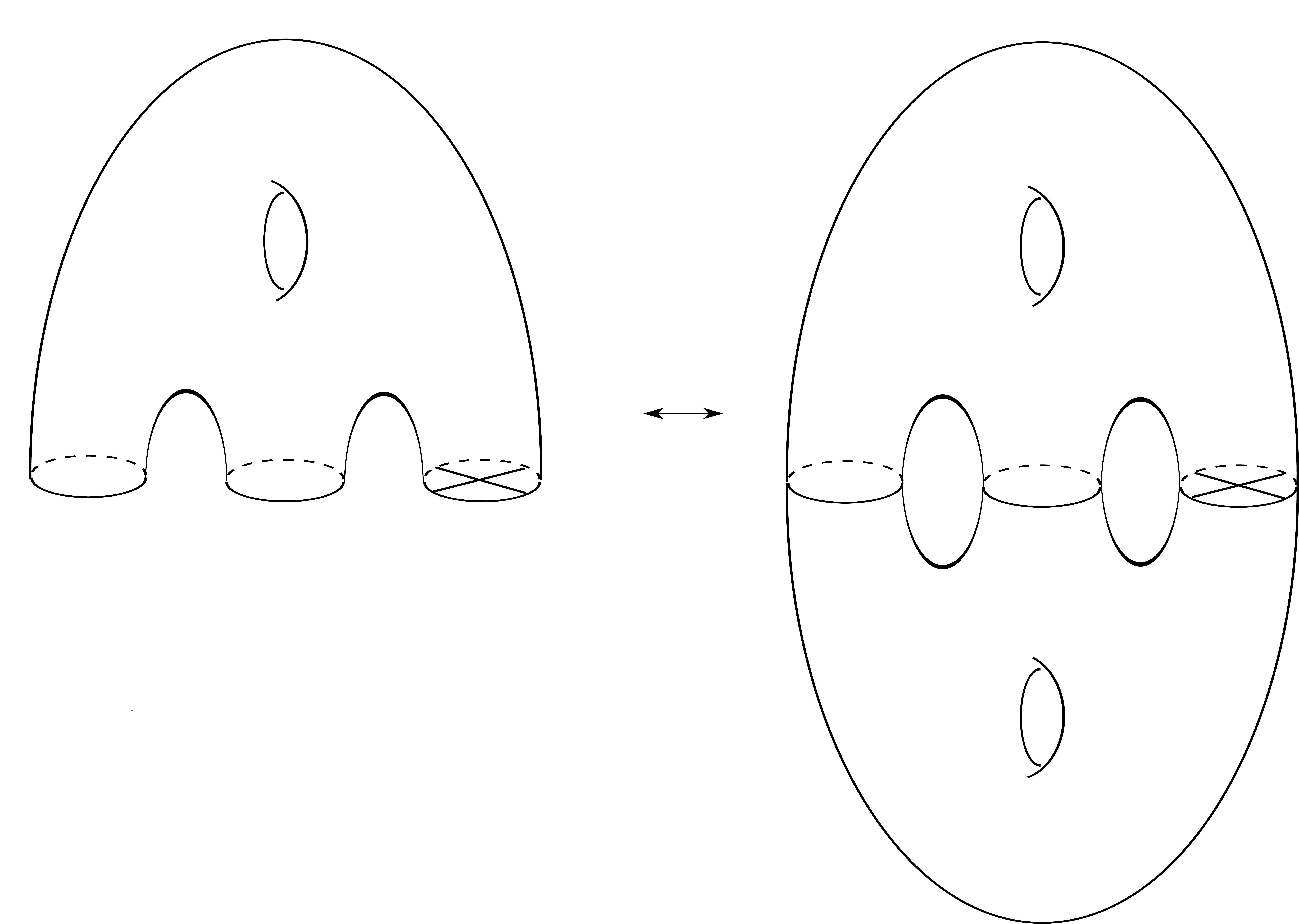}
\end{center}
\caption{ }
%\label{fig:def_image}
\end{figure}

\noindent
Since this doubling construction covers all topological types of orientation-reversing 
involutions $\si$ on $\Si$, for every symmetric surface $(\Si,\si)$ there is 
an oriented sh-surface $(\Si^b,c)$ which doubles to $(\Si,\si)$.
In general,  the topological type of such an sh-surface is not unique. 
There is a topologically unique oriented sh-surface~$(\Si^b,c)$
doubling to a symmetric surface~$( \Si,\si)$ if $\lr\si\!=\!0$, 
in which case $(\Si^b,c)$ has no crosscaps, or 
$|\si|_0\!\ge\!g( \Si)\!-\!1$,
in which case $(\Si^b,c)$ is either of genus at most~1 and has no crosscaps or
of genus~0 and has at most 2 crosscaps.\\

%\subsection{determinant lines}

\noindent
A \textsf{real Cauchy-Riemann operator on a real bundle pair $(V,\ti{\si})\!\lra\!(\Si,\si)$},  
where $(\Si,\si)$ is an oriented symmetric surface, is a linear map of the~form
\BE{CRdfn_e}\begin{split}
D=\bp\!+\!A\!: \Ga(\Si;V)^{\ti{\si}}
\equiv&\big\{\xi\!\in\!\Ga(\Si;V)\!:\,\xi\!\circ\!\si\!=\!\ti{\si}\!\circ\!\xi\big\}\\
&\hspace{.1in}\lra
\Ga_{\fJ}^{0,1}(\Si;V)^{\ti\si}\equiv
\big\{\ze\!\in\!\Ga(\Si;(T^*\Si,\fJ)^{0,1}\!\otimes_{\C}\!V)\!:\,
\ze\!\circ\!\tnd\si=\ti\si\!\circ\!\ze\big\},
\end{split}\EE
where $\bp$ is the holomorphic $\bp$-operator for some $\fJ\!\in\!\cJ_\si $
and a holomorphic structure in~$V$ and  
$$A\in\Ga\big(\Si;\Hom_{\R}(V,(T^*\Si,\fJ)^{0,1}\!\otimes_{\C}\!V) \big)^{\ti\si}$$ 
is a zeroth-order deformation term.  
A real Cauchy-Riemann operator on a real bundle pair
is Fredholm in the appropriate completions.
A continuous family of such Fredholm operators $D_t$ over a topological space~$\cH$  
determines a line bundle over~$\cH$, called \sf{the determinant line bundle of~$\{D_t\}$}
and denoted $\det D$;
see \cite[Section~A.2]{MS} and \cite{detLB} for a construction. More specifically, 
if $X,Y$ are Banach spaces and $D:X\rightarrow Y$ is a Fredholm operator, the
\textsf{determinant line} of $D$ is defined as 
$$\det (D)\equiv\La_{\R}^{\top}\ker(D) \otimes \big(\La^{\top}_{\R}\text{cok}(D)\big)^*.$$
A short exact sequence of Fredholm operators
\[\begin{CD}
0
@>>>X'@>>>X@>>>X''@>>>0 \\
@. @V V D' V@VV D V@VV D'' V@.\\
0@>>> Y'@>>>Y@>>>Y''@>>>0
\end{CD}\]
determines a canonical isomorphism
\BE{sum} \det (D)\cong \det (D')\otimes \det (D'').\EE
For a continuous family of Fredholm operators $D_t:X_t\ri Y_t$  parametrized by
a topological space $\cH$, the determinant lines $\det(D_t)$ form a line bundle
over~$\cH$. For a short exact sequence of such
families, the isomorphisms~(\ref{sum}) give rise to a canonical isomorphism
between determinant line bundles. 

\begin{rmk}\label{fam_rem}
Families of real Cauchy-Riemann operators often arise by pulling back data from
a target manifold by smooth maps as follows. 
Suppose $(X,\phi,J)$ is an almost complex manifold with an anti-complex  involution
$\phi\!:X\!\lra\!X$ and $(V,\ti\phi)\!\lra\!(X,\phi)$ is a real bundle pair.
Let $\na$ be a connection in $V$ and 
$$A\in\Ga\big(X;\Hom_{\R}(V,(T^*X,J)^{0,1}\otimes_{\C}\!V)\big)^{\ti\phi}.$$ 
For any real map $u\!:(\Si,\si)\!\lra\!(X,\phi)$ and $\fJ\!\in\!\cJ_{\si}$, 
let $\na^u$ denote the induced connection in $u^*V$ and
$$ A_{\fJ;u}=A\circ \prt_{\fJ} u\in\Ga(\Si;
\Hom_{\R}(u^*V,(T^*\Si,\fJ)^{0,1}\otimes_{\C}u^*V)\big)^{u^*\ti\phi}.$$
The homomorphisms
$$\bp_u^\na =\frac{1}{2}(\na^u+\fI\circ\na^u\circ\fJ), \,\,
D_u\equiv \bp_u^\na\!+\!A_{\fJ;u}\!: \Ga(\Si;u^*V)^{u^*\ti\phi}\lra
\Ga^{0,1}_{\fJ}(\Si;u^*V)^{u^*\ti\phi}$$
are real Cauchy-Riemann operators on $u^*(V,\ti\phi)\!\lra\!(\Si,\si)$
that form families of real Cauchy-Riemann operators over families of maps. We denote the determinant line bundle of such a family by~$\det D_{(V,\ti\phi)}$.
\end{rmk}

\noindent
Denote by $\cD_\si$ the group of orientation preserving diffeomorphisms of~$\Si$  
 commuting with the involution~$\si$.
If $(X,\phi)$ is a smooth manifold with an involution, $l\!\in\!\Z^{\ge0}$,
and $B\in H_2(X;\Z)$, let  
$$\fB_{g,l}(X,B)^{\phi,\si}\subset \fB_g(X)^{\phi,\si}\times \Si^{2l}$$
denote the space of real maps  $u\!:(\Si,\si)\!\lra\!(X,\phi)$ with $u_*[ \Si]_{\Z}=B$
and $l$ pairs of conjugate non-real marked distinct points.
We define
$$\cH_{g,l}(X,B)^{\phi,\si}=
\big(\fB_{g,l}(X,B)^{\phi,\si}\times\cJ_\si\big)/\cD_\si.$$
The action of $\cD_\si$ on $\cJ_{\Si}$ given by $h\cdot\fJ=h^*\fJ$
preserves~$\cJ_\si$; thus, the above quotient is well-defined. 
If $J$ is an almost complex structure on~$X$ such that $\phi^*J\!=\!-J$,
the moduli space of marked real $J$-holomorphic maps in the class $B\in H_2(X;\Z)$ 
is defined to be
$$\fM_{g,l}(X,J,B)^{\phi,\si}=
\big\{[u,(z_1,\si(z_1)),\ldots,(z_l,\si(z_l)),\fJ]\!\in\!\cH_{g,l}(X, B)^{\phi,\si}\!:~
\dbar_{J,\fJ}u\!=\!0\big\},
$$
 where $\dbar_{J,\fJ}$ is the usual Cauchy-Riemann operator with respect to the complex structures $J$ on $X$ and $\fJ$ on $\Si$. If $X$ is a point and $B$ is zero, we denote by
$$\cM_{g,l}^\si\equiv \fM_{g,l}(\pt,0)^{\id,\si}\equiv \cH_{g,l}(\pt, 0)^{\id,\si}$$
the moduli space of marked symmetric domains. There is a natural forgetful map
$$\mathfrak{f}:\cH_{g,l}(X,B)^{\phi,\si}\lra  \cM_{g,l}^\si.$$
The determinant line bundle of a family of real Cauchy-Riemann operators
$D_{(V,\ti\phi)}$ on 
$$\fB_{g,l}(X,B)^{\phi,\si}\times\cJ_\si$$ 
induced by a real bundle pair $(V,\wt\phi)\ri (X,\phi)$ as in Remark~\ref{fam_rem} descends 
to a line bundle over $\cH_{g,l}(X,B)^{\phi,\si}$, which we still denote by $\det D_{(V,\ti\phi)}$.

\begin{eg}\label{ex_tbdl}
If $\fc_\C$ denotes the standard conjugation on $\C$ and $(V,\wt\phi)=(\C,\fc_\C)\lra (\pt,\id)$, the induced family of operators $\dbar_\C\equiv D_{(\C,\fc_\C)}$ on $\cM_{g,l}^\si$ defines a line bundle 
$$\det \dbar_\C \lra \cM_{g,l}^\si.$$
If $(X,\phi)$ is an almost complex manifold with anti-complex involution $\phi$ and 
$$(V,\wt\phi)=(X\!\times\!\C,\phi\!\times\!\fc_\C)\lra (X,\phi),$$
then there is a canonical isomorphism 
$$\det D_{(\C,\fc_\C)}\approx\mathfrak{f}^*\dbar_\C$$
of line bundles over $\cH_{g,l}(X,B)^{\phi,\si}$.
%\begin{gather*}
%\xymatrix{\det D_{(\C,\fc_\C)}\ar[dr]  \ar[rr]^{\approx}&&\mathfrak{f}^*\dbar_\C \ar@{->}[ld]\\&\cH_{g,l}(X,B)^{\phi,\si}\,.}
%\end{gather*}
\end{eg}

\begin{rmk}\label{mfld_rem} 
For simplicity, we will assume that the action of $\cD_\si$ has no fixed points on the relevant
subspaces of $\fB_{g,l}(X,B)^{\phi,\si}\!\times\!\cJ_\si$.
This happens for example if sufficiently many marked points are added to~$\Si$.
In more general cases,  this issue can be avoided by working with 
Prym structures on Riemann surfaces; see~\cite{Loo}.
This assumption ensures that $\fM_{g,l}(X,J,B)^{\phi,\si}$ is a manifold 
if cut out transversely and thus has a first Stiefel-Whitney class.
Alternatively, if $g\!+\!2l\!\ge\!4$, the subspace of $\cM_{g,l}^\si$ 
consisting of $(\Si,\fJ)$ with non-trivial automorphisms is of codimension at least~2,
and so $\fM_{g,l}(X,J,B)^{\phi,\si}$ is a manifold outside of subspaces of codimension~2 
if cut out transversely and thus again  has a first Stiefel-Whitney class.
\end{rmk}

\noindent 
The following example shows that the orientability of a moduli space of symmetric half-surfaces is not sufficient for the orientability of the corresponding component of the moduli space of symmetric doubles.

\begin{figure}[htb]
\begin{center}
\leavevmode
\includegraphics[width=0.7\textwidth,height=150px]{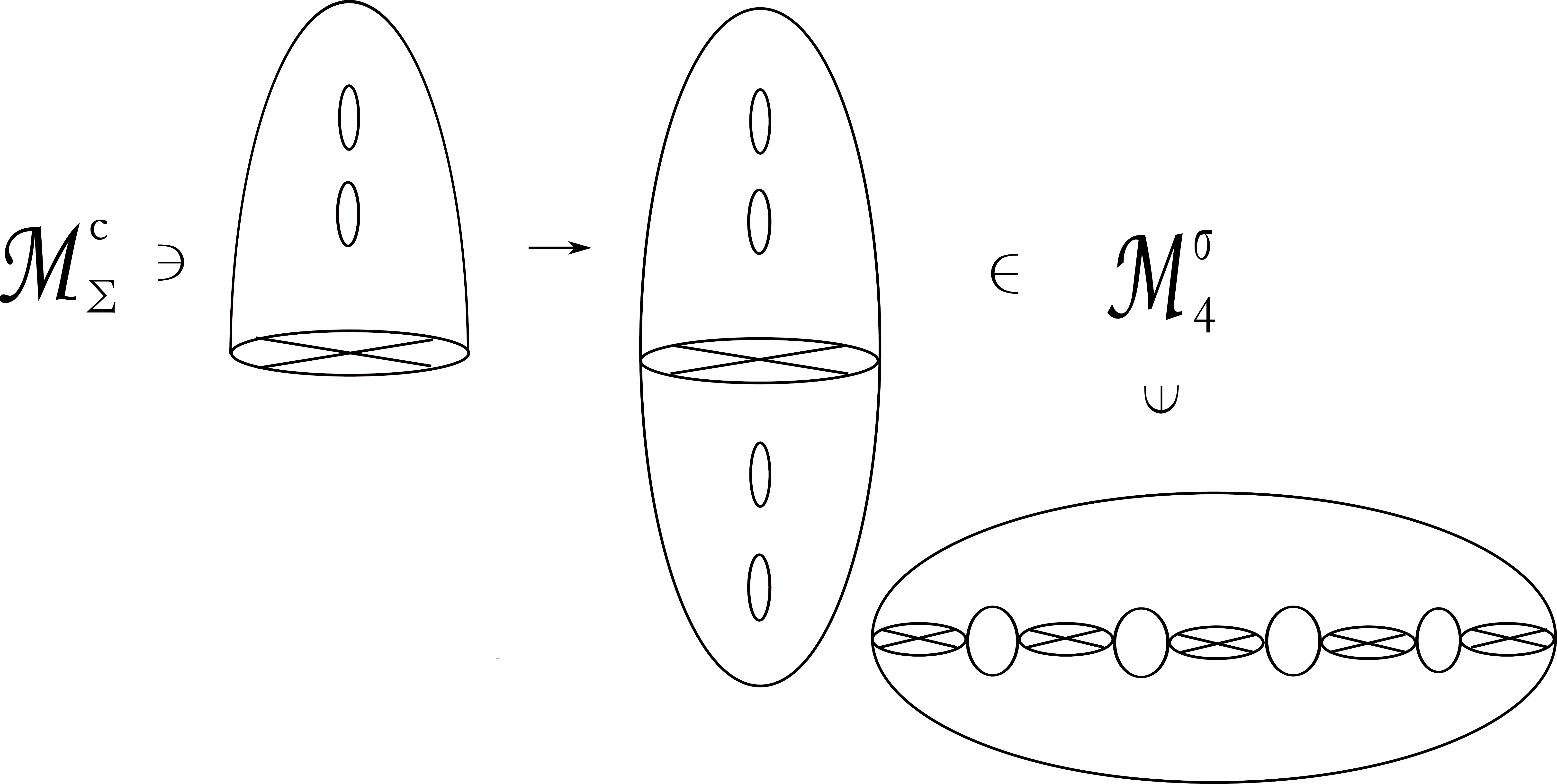}
\end{center}
\caption{  }
\label{fig:realspace}
\end{figure}

\begin{eg}\label{half_eg}
The moduli space $\cM_{\Si}^{c}$ of sh-surfaces $\Si$ of genus 2 with one boundary component and non-trivial involution (Figure \ref{fig:realspace}) is orientable by \cite[Lemma 6.1]{GZ} and \cite[Theorem]{IS}. The natural automorphisms of $\cM_{\Si}^{c}$  associated with real orientation-reversing diffeomorphisms of $\Si$ are   orientation-preserving by \cite[Corollary 2.2]{GZ2} and \cite[Lemma 6.1]{GZ}. On the natural double of $\Si$ (Figure \ref{fig:realspace}), which is a symmetric surface of genus 4 with an involution $\si$ without fixed locus, these diffeomorphisms correspond to flipping the surface across the crosscap. The real moduli space~$\cM^\si_4$, parametrizing such symmetric surfaces, is not orientable. A particular loop supporting its first Stiefel-Whitney class can be described as follows. By \cite[Theorem 1.2]{Nat}, every representative of a point in~$\cM^\si_4$ has 5 invariant circles which separate the surface. There is a real diffeomorphism~$h$ which fixes 3 of these circles and interchanges the other 2. By \cite[Corollary 2.2]{GZ2}, the mapping torus of $h$ defines a loop in $\cM^\si_4$ which pairs non-trivially with the first Stiefel-Whitney class of the moduli space. 

\end{eg}

\section{Topological preliminaries}
\label{signcomp_sec}

\noindent
In this section, we establish Proposition~\ref{bndlsplit_prp},
which relates the determinant lines of real Cauchy-Riemann operators 
on a real bundle pair and on its top exterior power.
This is the key statement used in the proof of Theorem~\ref{thm_maps}.

\begin{lmm}\label{lmm_repar}
Let $(\Si,\si)$ be a symmetric surface with fixed components $\Si^\si_1,\ldots,\Si^\si_m$ 
and $n\!\in\!\Z^+$ with $n\!\ge\!3$.
For every $a\!\in\!\pi_1(\tO(n))=\Z_2\!\times\!\Z_2$ and $i=1,\ldots,m$, 
there is a map $\psi_i\!:\Si\!\lra\!\U(n)$ such~that
\begin{enumerate}[label=$\bullet$,leftmargin=*]
\item $\psi(z)=\ov{\psi(z)}$,
\item $\psi$ is the identity outside of a small neighborhood of $\Si^\si_i$, and
\item $\psi_{|\Si^\si_i}=a\in \pi_1(\tO(n))$.
\end{enumerate}
\end{lmm}

\begin{proof}
Let $S^1\!\times\!(-2,2)$ be a $\si$-equivariant parametrization of a neighborhood~$U$ 
of $\Si^\si_i$ with $S^1\!\times\!0$ corresponding to~$\Si^\si_i$. 
Since the homomorphism $\pi_1(\tO(n))\!\lra\!\pi_1(\U(n))$ induced by the inclusion
is trivial, we can homotope~$a$ to the identity-valued constant map through 
maps $h_t\!:S^1\!\lra\!\U(n)$. 
We define~$\psi$ on~$U$ by   
$$\psi(\th,t)=
\begin{cases} 
h_t(\th), &\text{if}~ t\in[0,1];\\
I_n, &\text{if}~t\in[1,2);\\
\ov{\psi(\th,-t)},&\text{if}~t\in(-2,0];
\end{cases}$$
and extend it as the identity-valued constant map over $\Si\!-\!U$.
\end{proof}

\noindent
If $( \Si,\si)$ is  a symmetric surface and
$G\!:( \Si,\si)\!\lra\!( \Si,\si)$  is a real  diffeomorphism,
we define the mapping cylinder~$(M_G,\si_G)$ of $G$~by 
\begin{alignat*}{2}
M_G&=\bI\!\times\! \Si/\!\!\sim, &\qquad (1,z)&\sim\big(0,G(z)\big)
~~\forall\,z\!\in\! \Si,\\
\si_G\!:M_G&\lra M_G, &\qquad [s,z]&\lra\big[s,\si(z)\big]~~
\forall\,(s,z)\!\in\!\bI\!\times\! \Si.
\end{alignat*}

\begin{prp}\label{bndlsplit_prp0}
Let $( \Si,\si)$ be a symmetric surface,
$G\!:( \Si,\si)\!\lra\!( \Si,\si)$ be a real orientation-preserving diffeomorphism,
and $(W,\ti\phi)$ be a rank~$n$ real bundle pair over~$(M_G,\si_G)$.
If $n\!\ge\!3$, $W^{\ti\phi}\!\lra\!M_G^{\si_G}$ is orientable,
$w_2(W^{\ti\phi})\!=\!0$, and
$c_1(W)|_{ \Si_s}\!=\!0$ for any $s\!\in\!\bI$, there is an isomorphism
$$(W,\ti\phi)  \approx \La_{\C}^{\top}(W,\ti\phi)\oplus
(n\!-\!1)(M_G\!\times\!\C,\si_G\!\times\!\fc_{\C})$$
of real bundle pairs.
\end{prp}

\begin{proof}
\noindent
By \cite[Propositions~4.1,~4.2]{BHH}, there is an isomorphism of real bundle pairs
$$(W,\ti\phi)\approx \big(\bI\!\times\! \Si\!\times\!\C^n,
\id_{\bI}\!\times\!\si\!\times\!\fc_{\C^n}\big)/\!\!\sim_g,
\quad\hbox{where}~~~
(1,z,v)\sim_g\big(0,G(z),g(z)v\big)\quad\forall\,(z,v)\!\in\! \Si\!\times\!\C^n,$$
for some $g\!: \Si\!\lra\!\tU(n)$ such that $g(\si(z))\!=\!\ov{g(z)}$ 
for all $z\!\in\!\Si$.\\

\noindent
We first show that the above isomorphism can be chosen so that 
$$g_{|\Si^\si_i}:\Si^\si_i\lra O(n)$$ 
is homotopic to the identity-valued constant map~$\Id$ on each fixed component~$\Si^{\si}_i$
for $i=1,\ldots,m$.  
The map~$G$ defines a permutation on $\{\Si^\si_i\}$. 
Every cycle $(i_1,\dots,i_k)$  in this permutation  defines a connected component $C$ 
of $M_G^{\si_G}$.  Since $W^{\ti\phi}$ is spin and $\rk(W^{\ti\phi})\geq 3$, the bundle $W^{\ti\phi}|_C$ is trivial. 
Thus, 
\BE{spincond_e}\sum_{l=1}^k \big[g_{|\Si^\si_{i_l}}\big] =1 \in \pi_1(\tO(n)).\EE
For $j=2,\dots, k$, let 
\BE{repar_e}
\big[a_{i_j}\big]=\sum_{l=1}^{j-1} \big[g_{|\Si^\si_{i_l}}\big]  \in \pi_1(\tO(n))\EE
and $\psi_{i_j}:\Si\!\lra\!\U(n)$ be the map constructed in Lemma~\ref{lmm_repar} 
corresponding to $(i,[a])\!=\!(i_j,[a_{i_j}])$. Let
$$\Psi=\psi_{i_k}\cdot\ldots\cdot\psi_{i_2}\!:\Si\lra\U(n).$$
There is a real bundle pair isomorphism over $(M_G,\si_G)$
\begin{gather*} 
\big(\bI\!\times\! \Si\!\times\!\C^n,\id_{\bI}\!\times\!\si\!\times\!\fc_{\C^n}\big)/\!\!\sim_g
\lra \big(\bI\!\times\! \Si\!\times\!\C^n,
\id_{\bI}\!\times\!\si\!\times\!\fc_{\C^n}\big)/\!\!\sim_{\ti g},\qquad
(s,z,v)\lra (s,z,\Psi(z)v),\\
\hbox{where}\qquad  \ti g(z)=\Psi(G(z))g(z)\Psi^{-1}(z). 
\end{gather*}
Since $\Psi_{|\Si^{\si}_{i_j}}=\psi_{i_j|\Si^{\si}_{i_j}}$ for $j=2,\ldots,k$
and $G(\Si^{\si}_{i_j})=\Si^{\si}_{i_{j+1}}$ for $j=1,\dots,k\!-\!1$,
$$\big[\ti g_{|\Si^{\si}_{i_j}}\big] = \big[a_{i_{j+1}}\big]
+ \big[g_{|\Si^{\si}_{i_j}}\big] 
+ \big[a_{i_j}\big]=1\in\pi_1(O(n))\qquad \forall~j=2,\ldots,k\!-\!1;$$
the second equality follows from~\eref{repar_e}. 
Furthermore, since $\Psi_{|\Si^{\si}_{i_1}}=\Id$ and $G(\Si^{\si}_{i_k})=\Si^{\si}_{i_1}$, 
$$\big[\ti g_{|\Si^{\si}_{i_1}} \big]= \big[a_{i_{2}}\big]+\big[g_{|\Si^{\si}_{i_1}}\big]=1\in\pi_1(O(n))
\quad\hbox{and}\quad
\big[\ti g_{|\Si^{\si}_{i_k}}\big] = \big[ g_{|\Si^{\si}_{i_k}}\big]+\big[a_{i_k}\big]=1\in\pi_1(O(n))$$
by~\eref{repar_e}  and by~\eref{spincond_e}, respectively.\\

% Thus, without loss of generality we assume $$ g_{|\Si^\si_i}:\Si^\si_i\lra O(n)
%$$ is the identity map on each fixed component $\Si^{\si}_i$,   $i=1,\dots, m$.\\
% and $\sim_{\ti g}$ is defined to be 
%$(1,z,v)\sim_{\ti g}\big(0,G(z),\Psi(G(z))g(z)\Psi^{-1}(z)v)\quad\forall\,(z,v)\!\in\! \Si\!\times\!\C^n.$

\noindent
Let $e_1,\ldots,e_n$ be the standard coordinate basis for~$\C^n$.
We define a vector space isomorphism
\begin{gather*}
\Xi\!:\C^n\lra \La^{\top}_{\C}\C^n \oplus \C^{n-1} \qquad\hbox{by}\\
e_1\lra (e_1\!\w\!\ldots\!\w\!e_n,0), \qquad
e_i\lra(0,e_{i-1})~~~\forall\,i\!=\!2,\ldots,n.
\end{gather*}
In particular, the composition
$$\La_{\C}^{\top}\C^n\stackrel{\La_{\C}^{\top}\Xi}{\lra}
\La_{\C}^{\top}\big(\La^{\top}_{\C}\C^n\!\oplus\!\C^{n-1}\big)
=\La^{\top}_{\C}\C^n\otimes \La^{\top}_{\C}\C^{n-1}
\lra \La^{\top}_{\C}\C^n,$$
where the last map sends $w\!\otimes\!e_1\!\w\!\ldots\!\w\!e_{n-1}$
to~$w$, is the identity.
Define
$$f\!: \Si\lra\tU(n) \qquad\hbox{by}\qquad
f(z)\Xi\big(g(z)v\big)=
\left(\begin{array}{cc}\det g(z)& 0\\ 0&\bI_{n-1}\end{array}\right)\Xi(v)
\quad\forall\,(z,v)\!\in\! \Si\!\times\!\C^n.$$
In particular, $f(z)\!\in\!\SU(n)$ and $f(\si(z))\!=\!\ov{f(z)}$ 
for all $z\!\in\! \Si$.
It is sufficient to show that $f$ is homotopic to the constant map~$\Id$
subject to these conditions.\\

\noindent
Let $(\Si^b,c)$ be an oriented sh-surface that doubles to~$(\Si,\si)$. 
Since the map $g$ can be chosen to be homotopic to~$\Id$ on each fixed 
component~$\Si_i^{\si}$, 
the map $f\!:(\prt\Si^b)_i\!\lra\!\tO(n)$ is homotopic to~$\Id|_{(\prt\Si^b)_i}$ 
on each boundary component $(\prt\Si^b)_i$ with $|c_i|\!=\!0$.
For each boundary component $(\prt\Si^b)_i$ with $|c_i|\!=\!1$,
the map $f_{|(\prt\Si^b)_i}\!:(\prt\Si^b)_i\!\lra\!\SU(n)$ is homotopic to~$\Id$;
see \cite[Lemma~2.2]{Teh}.
In both cases, the homotopies are through maps
$f_t\!:(\prt\Si^b)_i\!\lra\!\SU(n)$ such that  $f_t(\si(z))\!=\!\ov{f_t(z)}$ 
for all $z\!\in\!(\prt\Si^b)_i$.
They extend over $\Si^b$ as follows. 
Let $S^1\!\times\!\bI\!\lra\!U$ be a parametrization of a (closed) neighborhood~$U$ 
of  $(\prt\Si^b)_i\!\subset\!\Si^b$ with coordinates~$(\th,s)$ and~define
$$G_t:\Si^b\lra\SU(n) \qquad\hbox{by}\quad
G_t(z)=\begin{cases}
f_{(1-s)t}(\th)\cdot f^{-1}(\th),&\text{if}~ z=(\th,s)\in U\approx S^1\!\times\!\bI;\\
I_n,& \text{if}~ z\in \Si^b\!-\!U.
\end{cases}$$
Since $G_t((\th,1))=I_n$ for all~$t$, this map is continuous. 
Moreover, $G_0(z)=I_n$ for all $z\in \Si^b$ and 
$$G_t((\th, 0))=f_{t}(\th)\cdot f^{-1}(\th)$$ 
is a homotopy between $\Id$ and $f^{-1}$. Thus, $H_t=G_t\cdot f$ is a homotopy 
over $\Si^b$ extending~$f_t$.\\

\noindent
By the above, we may assume that $f$ is the constant map~$\Id$ on the boundary of~$\Si^b$. 
Choose arcs in $\Si^b$ with endpoints on $\prt\Si^b$ which cut $\Si^b$ into a disk. 
\begin{figure*}[htb]
\begin{center}
\advance\topskip-10cm
%\leavevmode
\includegraphics[width=0.3\textwidth,height=150px]{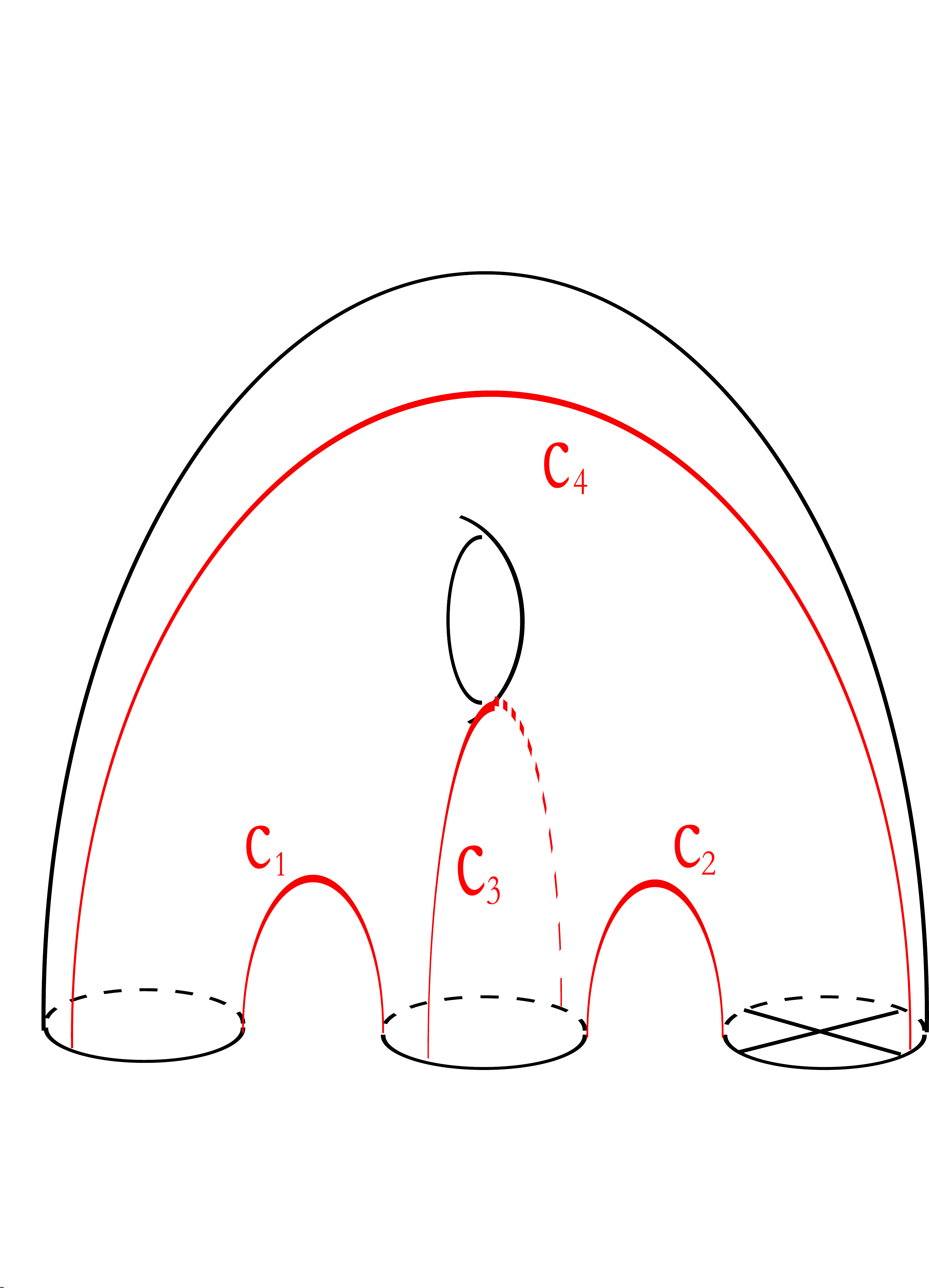}
\end{center}
\caption*{The arcs $c_1,\dots, c_4$ cut $\Si^b$ to a disk.}
%\label{fig:def_image}
\end{figure*}
Each such arc defines an element of $\pi_1(\SU(n),I_n)=0$.
Thus, we can homotope~$f$ to~$\Id$ over the arcs while keeping it fixed at the endpoints. 
Similarly to the above, this homotopy extends over~$\Si^b$. 
Thus, we may assume that $f$ is the constant map~$\Id$ over the boundary of 
the disk obtained from cutting~$\Si^b$ along the~arcs. 
Since $\pi_2(\SU(n),I_n)=0$,
$$f\!:(D^2,S^1)\lra \big(SU(n),I_n\big)$$
can be homotoped to~$\Id$ as a relative map.
Doubling such a homotopy~$f_t$  by the requirement that $f_t(\si(z))\!=\!\ov{f_t(z)}$ 
for all $z\!\in\!\Si$, we obtain the desired homotopy from $f$ to~$\Id$ over all of~$\Si$.
\end{proof}

\begin{prp}\label{bndlsplit_prp}
Let $( \Si,\si)$ be a symmetric surface,
$G\!:( \Si,\si)\!\lra\!(\Si,\si)$ be a real orientation-preserving diffeomorphism,
and $(W,\ti\phi)$ be a rank~$n$ real bundle pair over~$(M_G,\si_G)$.
If $n\!\ge\!2$, $W^{\ti\phi}\!\lra\!M_G^{\si_G}$ is orientable,
and $w_2(W^{\ti\phi})\!=\!0$, there is an isomorphism
$$(W,\ti\phi)\oplus \big(\La_{\C}^{\top}(W,\ti\phi)\big)^*
\approx (n\!+\!1)(M_G\!\times\!\C,\phi_G\!\times\!\fc_{\C})$$
of real bundle pairs.
\end{prp}

\begin{proof}
Applying Proposition~\ref{bndlsplit_prp0}, we obtain
\begin{equation*}\begin{split}
(W,\ti\phi)\oplus \big(\La_{\C}^{\top}(W,\ti\phi)\big)^*
&\approx \La_{\C}^{\top}\big((W,\ti\phi)\!\oplus\!\big(\La_{\C}^{\top}(W,\ti\phi)\big)^*\big)
\oplus n(M_G\!\times\!\C,\phi_G\!\times\!\fc_{\C})\\
&= \La_{\C}^{\top}(W,\ti\phi)\otimes\big(\La_{\C}^{\top}(W,\ti\phi)\big)^*
\oplus n(M_G\!\times\!\C,\phi_G\!\times\!\fc_{\C})\\
&= (n\!+\!1)(M_G\!\times\!\C,\phi_G\!\times\!\fc_{\C})\,.
\end{split}\end{equation*}
This establishes the claim.
\end{proof}

\noindent
The importance of these propositions to the orientability problem lies in the implication that they give rise to isomorphisms of the determinant  bundles of $\dbar$-operators on the two sides, inducing equality on their first Stiefel-Whitney classes. This equality may still hold even without the splittings of the bundles provided by the above propositions; see the appendix.

\section{Proofs of main statements}
\label{mainpf_sec}
 
\noindent
In this section, we use Proposition~\ref{bndlsplit_prp} to show that the first Stiefel-Whitney classes of $\cM_{g,l}^\si$ and of the determinant line bundle for the trivial rank~1 real bundle pair over it are the same; see Proposition~\ref{prp_domains}. 
This result is also obtained in~\cite{Remi}; see Corollaire~1.2, Corollaire~1.1, Proposition~1.4,
Lemme~1.3, and Lemme~1.4 in~\cite{Remi}.
We then use Proposition~\ref{prp_domains} along with Proposition~\ref{bndlsplit_prp} to establish Theorem~\ref{thm_maps}. We conclude this section by deducing Corollary~\ref{CIorient_crl} from Theorem~\ref{thm_maps}.

\begin{prp}\label{prp_domains} 
Let $g,l\in \Z^{\ge0}$ be such that $g\!\ge\!2$ and $g+2l\ge4$. 
If $(\Si,\si)$ is a genus $g$ symmetric surface,
$$w_1(\cM_{g,l}^{\si}) = w_1(\det\dbar_\C).$$
\end{prp}
 
\begin{proof}
Via the Kodaira-Spencer map, $T_{[\fJ]}\cM^{\si}_{g}$ is canonically isomorphic 
to $H_{\fJ}^1(\Si;T\Si)^{\sigma}$; see \cite[Section~3.1.2]{Melissa}.
By Serre duality \cite[p153]{GH}, there is a canonical isomorphism
$$H_{\fJ}^1(\Si;T \Si)^{\sigma}\approx 
\big(H_{\fJ}^0( \Si;T^*\Si^{\otimes2})^{\sigma}\big)^*.\footnotemark$$
Since the genus of $\Si$ is at least~2 under our assumptions, 
 $$H_{ \fJ}^0( \Si;T^* \Si^{\otimes2})^{\sigma}\cong \ker\dbar_{(T^* \Si,\text{d}\sigma^*)^{\otimes2}}.$$
The forgetful map 
$$f:\cM_{g,l}^\si\lra \cM_g^\si$$
with fiber isomorphic to an open subset of $\Si^l$, 
determined by the positions of the first elements in the $l$~pairs of conjugate points,  
induces an isomorphism
$$\La^\top_\R (T\cM_{g,l}^\si)\approx 
\La^\top_\R (f^{\text{Vert}})\otimes\La^\top_\R (f^*T\cM_g^\si).$$
Since the elements of  $\cD_\si$ preserve the orientation of $\Si$, the bundle  $\La^\top_\R( f^{\text{Vert}})$ is orientable. Thus,
$$w_1(\cM_{g,l}^\si)=w_1(f^*\ker\dbar_{(T^* \Si,\tnd\si^*)^{\otimes2}}).$$
\footnotetext{The real part of the Serre duality identifies the spaces of invariant sections
on one side with the space of anti-invariant sections on the other; the latter is isomorphic
to the space of invariant sections by multiplication by~$\fI$.}
\noindent
Let $\gm$ be a loop in $\cM_{g,l}^{\si}$. Under our assumptions, $\gm$ can be taken in the smooth locus  and thus lifts to a mapping torus $(M_G,\si_G)$ for some real diffeomorphism 
$G\!: (\Si,\si)\!\lra\!(\Si,\si)$. 
Let
$$(W,\ti\phi)=\bI\!\times\!T\Si\big/\!\!\sim\lra (M_G,\si_G),\qquad 
(1,v)\sim(0,\tnd G(v))~\forall\,v\!\in\!T\Si,$$
with the complex structure in the fiber of $W$ over $s\!\times\!\Si$ being~$\fJ_s$.
Since 
$$w_1((W,\ti\phi)\oplus (W,\ti\phi)\oplus (W,\ti\phi)^{*\otimes2})=0=w_2((W,\ti\phi)\oplus (W,\ti\phi)\oplus (W,\ti\phi)^{*\otimes2}),$$
 by Proposition~\ref{bndlsplit_prp}
$$(W,\ti\phi)\oplus (W,\ti\phi)\oplus (W,\ti\phi)^{*\otimes2}
\approx 3(M_G\!\times\!\C,\si_G\!\times\!\fc_{\C}).$$
Since   the indices of    $\dbar$-operators
on $2(W,\wt\phi)$ and $2( M_G\!\times\!\C,\si_G\!\times\!\fc_\C)$ are canonically oriented,
$$
w_1(\det\dbar_{(T^* \Si, \sigma^*)^{\otimes2}})=w_1(\det\dbar_\C),
$$ 
proving the claim.
\end{proof}

\begin{proof}[{\bf \emph{Proof of Theorem \ref{thm_maps}}}] 
First assume that $l$ is sufficiently large so that $\cM_{g,l}^{\si}$ is a manifold;
see Remark~\ref{mfld_rem}.
If the moduli space $\fM_{g,l}(X,B)^{\phi,\si}$ is cut transversely, the forgetful map 
$$\mathfrak{f}: \fM_{g,l}(X,B)^{\phi,\si}\lra\cM_{g,l}^{\si}$$ 
induces the equality of   first Stiefel-Whitney classes
$$w_1(\fM_{g,l}(X,B)^{\phi,\si})=w_1(\det D_{(TX,\tnd\phi)})+
\mathfrak{f}^*w_1(\cM_{g,l}^{\si}).$$
Thus, it suffices to show that
$$w_1(\det D_{(TX,\tnd\phi)})=n\,\mathfrak{f}^*w_1(\cM_{g,l}^{\si})$$ 
over $\cH_{g,l}(X,B)^{\phi,\si}$.\\

\noindent
By \eref{sum}, there is a canonical isomorphism
$$\det D_{(TX\oplus 2L, \tnd\phi\oplus 2\ti\phi)}
\approx\det D_{(TX, \text{d}\phi)}\otimes (\det D_{(L,\ti\phi)})^{\otimes2}$$ 
and thus
$$w_1(\det D_{(TX\oplus 2L,\tnd\phi\oplus 2\ti\phi)})
=w_1(\det D_{(TX,\tnd\phi)}).$$
Let $\gm$ be a loop in $\fM_{g,l}(X,B)^{\phi,\si}$.
Under the assumption of Remark~\ref{mfld_rem}, the projection 
$$\fB_{g,l}(X,B)^{\phi,\si}\times \cJ_\si\lra \cH_{g,l}(X,B)^{\phi,\si}$$ 
admits local slices. Thus, there exists a path $\wt\gm_t=(u_t,j_t)$ in
$\fB_g(X,B)^{\phi,\si}\!\times\!\cJ_\si$ lifting $\gm$ and 
a real diffeomorphism  $G\!\in\!\cD_\si$  such that $\wt\gm_1=G\cdot\wt\gm_0$.
Let $(M_G,\si_G)$ be the corresponding mapping torus. 
By Proposition~\ref{bndlsplit_prp}, 
\BE{eq_iso}
\ev^*(TX\oplus 2L,\tnd\phi\oplus 2\ti\phi)\oplus 
 \ev^*(\La_\C^{\top}(TX\oplus 2L,\tnd\phi\oplus 2\ti\phi))^*\approx 
 (n\!+\!3)\,(M_G\!\times\!\C,\si_G\!\times\!\fc_\C),\EE
where $\ev\!:M_G\!\lra\!X$ is the natural evaluation map determined by~$\wt\gm_t$.
Note that 
$$\La_\C^{\top}(TX\oplus 2L,\tnd\phi\oplus 2\ti\phi)
\approx\big((L,\ti\phi)^{\otimes 2}\big)^{\otimes 2}$$
as real bundle pairs over $(X,\phi)$.
By Proposition~\ref{bndlsplit_prp}, 
$$2\,\ev^*(L,\ti\phi)^{\otimes 2} \oplus 
\big( \La_\C^{\top}\big(2\,\ev^*(L,\ti\phi)^{\otimes 2}\big)\big)^*
\approx  3(M_G\!\times\!\C, \si_G\!\times\!\fc_\C).$$
Since the determinant line bundles of $\dbar$-operators
on $2\,\ev^*(L,\ti\phi)^{\otimes 2}$ and $2(M_G\!\times\!\C,\si_G\!\times\!\fc_\C)$ 
are canonically oriented,
$$w_1\big(\det\dbar_{(\La_\C^{\top}(TX\oplus 2L,\tnd\phi\oplus 2\ti\phi))^*}\big)
=w_1\big(\det\dbar_{((L,\ti\phi)^{\otimes 4})^*}\big)
=w_1(\det\dbar_\C).$$ 
By (\ref{eq_iso}),
$$w_1\big(\det\dbar_{(TX\oplus 2L,\tnd\phi\oplus 2\ti\phi)}\big)
+w_1(\det\dbar_{\La_\C^{\top}(TX\oplus 2L,\tnd\phi\oplus 2\ti\phi)}) 
= (n\!+\!3)\, w_1(\det\dbar_\C),$$
and thus
$$w_1(\det\dbar_{(TX\oplus 2L,\tnd\phi\oplus 2\ti\phi)})
=n\,(w_1(\det\dbar_\C))=n\,\mathfrak{f}^*w_1(\cM_{g,l}^\si),$$
where the last equality holds by Proposition~\ref{prp_domains}.\\
%This establishes Theorem~\ref{thm_maps}.

\noindent
As in the proof of Propositions~\ref{prp_domains}, the vertical bundle of the map
$$f\!: \fM_{g,l}(X,B)^{\phi,\si} \lra \fM_{g,l'}(X,B)^{\phi,\si}$$
forgetting the last $l\!-\!l'$ marked points is orientable.
Thus,  $\fM_{g,l'}(X,B)^{\phi,\si}$ is as orientable as $\fM_{g,l}(X,B)^{\phi,\si}$.
The condition $g\!+\!2l\!\ge\!4$ in~(2) of Theorem~\ref{thm_maps} ensures that 
$w_1(\cM_{g,l}^{\si})$ is defined.  
\end{proof}

\begin{proof}[{\bf\emph{Proof of Corollary~\ref{CIorient_crl}}}]
The involutions $\tau_n$ on~$\P^{n-1}$ and $\eta_{2n}$ on~$\P^{2n-1}$
naturally lift to involutions~$\ti\tau_n$ on $\cO_{\P^{n-1}}(1)$
and~$\ti\eta_{2n}$ on $2\cO_{\P^{2n-1}}(1)$ 
so that the usual Euler sequences for~$\P^{n-1}$ and $\P^{2n-1}$
become short exact sequences of real bundle pairs: 
\begin{gather*}
0\lra (\P^{n-1}\!\times\!\C,\tau_n\!\times\!\id_{\C})\lra
n\big(\cO_{\P^{n-1}}(1),\ti\tau_n\big)\lra (T\P^{n-1},\tnd\tau_n)\lra0,\\
0\lra (\P^{2n-1}\!\times\!\C,\eta_{2n}\!\times\!\id_{\C})\lra
n\big(2\cO_{\P^{2n-1}}(1),\ti\eta_{2n}\big)\lra (T\P^{2n-1},\tnd\eta_{2n})
\lra0\,.
\end{gather*}
If $X_{n;\a}\!\subset\!\P^{n-1}$ and $X_{2n;\a}\!\subset\!\P^{2n-1}$ 
are complete intersections preserved by the involutions~$\tau_n$ and~$\eta_{2n}$,
respectively, the sequences
\begin{gather*}
0\lra (TX_{n;\a},\tnd\tau_{n;\a})\lra  
(T\P^{n-1},\tnd\tau_{n})\big|_{X_{n;\a}}\lra  
\bigoplus_{i=1}^k\big(\cO_{\P^{n-1}}(1),\ti\tau_n\big)^{\otimes a_i}\big|_{X_{n;\a}}
\lra 0,\\
\begin{split}
0&\lra (TX_{2n;\a},\tnd\eta_{2n;\a})\lra  
(T\P^{2n-1},\tnd\eta_{2n})\big|_{X_{2n;\a}}\\
&\lra  \bigoplus_{a_i\in2\Z}
\big(\La_{\C}^{\top}(\cO_{\P^{2n-1}}(2),\ti\eta_{2n})\big)^{\otimes (a_i/2)}\big|_{X_{2n;\a}}
\oplus \bigoplus_{a_i'\not\in2\Z} 
(2\cO_{\P^{2n-1}}(a_i'),\ti\eta_{2n})\big)\big|_{X_{2n;\a}}
\lra 0
\end{split}
\end{gather*}
are also  short exact sequences of real bundle pairs.\footnote{In the second case,
the odd degrees~$a_i$ come in pairs; 
the second sum is taken over one $a_i'\!=\!a_i$ for each such pair.}
Thus, under the assumptions of Corollary~\ref{CIorient_crl},
\begin{equation*}\begin{split}
\La_{\C}^{\top}\big(TX_{n;\a},\tnd\tau_{n;\a}) &\approx 
\big(\big(\cO_{\P^{n-1}}(1),\ti\tau_n\big)^{\otimes ((n-|\a|)/2)}\big|_{X_{n;\a}}
\big)^{\otimes2}, \\
\La_{\C}^{\top}\big(TX_{2n;\a},\tnd\eta_{2n;\a}) &\approx 
\big(\big(\La_{\C}^{\top}(\cO_{\P^{2n-1}}(2),\ti\eta_{2n})\big)^{\otimes ((2n-|\a|)/4)}
\big|_{X_{2n;\a}}\big)^{\otimes2}\,,
\end{split}\end{equation*}
where $|\a|\!=\!a_1\!+\!\ldots\!+\!a_k$.
We denote the rank~1 real bundle pairs being squared above (before the square is taken)
by~$(L_{\tau},\ti\phi_{\tau})$ and $(L_{\eta},\ti\phi_{\eta})$.
Since $w_1(X_{n;\a}^{\tau_n})\!=\!0$ under the assumptions of 
Corollary~\ref{CIorient_crl}(1) and $X_{2n;\a}^{\eta_{2n}}\!=\!\eset$,
$$w_2(X_{n;\a}^{\tau_n})=\bigg(\binom{n}{2}-\sum_{i<j}a_ia_j\bigg)x^2
=\bigg(\frac{n\!-\!|\a|}{2}\bigg)^2x^2
=w_1(L_{\tau}^{\ti\phi_{\tau}})^2\,,\quad
w_2(X_{2n;\a}^{\eta_{2n}})=0=w_1(L_{\eta}^{\ti\phi_{\eta}})^2,$$
where $x$ is the restriction of the generator of $H^1(\R\P^{n-1};\Z_2)$ to 
$X_{n;\a}^{\tau_n}$ in the first case;
the middle equality in the first case  above follows from 
the numerical assumptions in Corollary~\ref{CIorient_crl}(1). 
Therefore,  $(X_{n;\a},\om_{n;\a},\tau_{n;\a})$ and $(X_{2n;\a},\om_{2n;\a},\eta_{2n;\a})$
are real-orientable in the sense  of Definition~\ref{realorient_dfn}
under the assumptions in~(1) and~(2), respectively,
of Corollary~\ref{CIorient_crl}.
The claim thus follows from Theorem~\ref{thm_maps} if the genus of~$\Si$ is at least~2
and from \cite[Theorems~1.1,~1.2]{GZ2} otherwise.
\end{proof}

\section{Examples: real Hurwitz covers}
\label{HurCov_sec}

\noindent
For $d\!\in\!\Z^+$, a \sf{degree~$d$ Hurwitz cover} of a closed connected Riemann surface~$(\Si_0,\fJ_0)$ is a holomorphic map $u\!:(\Si,\fJ)\!\lra\!(\Si_0,\fJ_0)$,
where $(\Si,\fJ)$ is another connected Riemann surface, such that 
$$u_*[\Si]_{\Z}=d[\Si_0]_{\Z}\in H_2(\Si_0;\Z).$$
By Riemann-Hurwitz \cite[p218]{GH}, such a map $u$ has 
\BE{RHdfn_e} 2b\equiv 2(d-1+g-dg_0)\,,\EE
branched points, counting multiplicity.
If $\si_0$ is an anti-holomorphic involution on~$(\Si_0,\fJ_0)$, 
a \sf{real Hurwitz cover} of $(\Si_0,\si_0,\fJ_0)$ is a real holomorphic map 
\BE{RHurmap_e} u\!:(\Si,\si,\fJ)\lra(\Si_0,\si_0,\fJ_0),\EE
i.e.~$u\!\circ\!\si\!=\!\si_0\!\circ\!u$.
In this section, we show that many moduli spaces 
$\fM(\Si_0,d;\fJ_0)^{\si_0,\si}$ of Hurwitz covers,
with a fixed complex target and a fixed topological domain,
are orientable; see Theorem~\ref{HurCov_thm}.

\begin{eg}\label{HurCov_eg1}
Let $(\Si_0,\si_0,\fJ_0)$ be a genus~$g_0$ symmetric Riemann surface such that 
$\Si_0^{\si_0}\!=\!\eset$ and $(\Si,\si)$ be a genus~$g$ symmetric surface.
If there exists a degree~$d$ Hurwitz cover as in~\eref{RHurmap_e}, then
$$d-1+g-dg_0\in 2\Z.$$
In particular, if $g$ is even, then $g_0$ is even and $d$ is odd.
\end{eg}

\begin{proof}
(1) Suppose first that $g_0$ is odd.
By \cite[Theorem~1.2]{Nat}, there are two disjoint circles \hbox{$C_1',C_2'\!\subset\!\Si_0$}
that are preserved by~$\si_0$ and split~$\Si_0$ into bordered surfaces
interchanged by~$\si_0$.
Similarly to the case $\Si_0\!=\!\bT$, these two circles can be replaced by two circles
$C_1,C_2\!\subset\!\Si_0$ that are interchanged by~$\si_0$ and still split~$\Si_0$ 
into bordered surfaces~$\Si_0^+$ and~$\Si_0^-$ interchanged by~$\si_0$.
The preimages of~$\Si_0^+$ and~$\Si_0^-$ split~$\Si$ into bordered surfaces
$\Si_1^+,\ldots,\Si_k^+$ and $\Si_1^-,\ldots,\Si_k^-$, which are interchanged by~$\si$.
Let $m_i^+$ be the number of boundary components of~$\Si_i^+$.
The set of boundary components of all these surfaces can be grouped into quadruples: 
pairs of them
are identified and mapped to~$C_1$ and conjugate pairs of them are mapped to~$C_2$ 
by~$\si$.
Thus, the Euler characteristic of~$\Si$, 
$$\chi(\Si)=2\big(\chi(\Si_1^+)+\ldots+\chi(\Si_k^+)\big)
=4\big(1\!-\!g(\Si_1^+)+\ldots+1\!-\!g(\Si_k^+)\big)
-2\big(m_1^++\ldots+m_k^+\big),$$
is divisible by~4.
This establishes the claim for $g_0$ odd.\\

\noindent
(2) Suppose that $g_0$ is even.
By \cite[Theorem~1.2]{Nat}, there is a circle $C\!\subset\!\Si_0$
which is preserved by~$\si_0$ and splits~$\Si_0$ 
into bordered surfaces~$\Si_0^+$ and~$\Si_0^-$ interchanged by~$\si_0$.
The preimages of~$\Si_0^+$ and~$\Si_0^-$ split~$\Si$ into bordered surfaces
$\Si_1^+,\ldots,\Si_k^+$ and $\Si_1^-,\ldots,\Si_k^-$, which are interchanged by~$\si$.
The set of boundary components of all these bordered surfaces 
can be grouped into sets of two types:
\begin{enumerate}[leftmargin=8mm]
\item[($-$)] pairs, in which each element corresponds to a circle in $\Si$ preserved by~$\si$;
\item[($+$)] quadruples, in which two elements correspond to a circle in $\Si$
not preserved by~$\si$
and the other two elements correspond to its image under~$\si$.
\end{enumerate}
We denote the number of pairs of the first type by $m^-$ and the number of
quadruples of the second type by~$m^+$.
The restriction of~$u$ to a boundary in~($-$) is a map $S^1\!\lra\!S^1$
commuting with the antipodal involution and must be of odd degree.
The restrictions of~$u$ to boundaries in~($+$) in the same quadruple are maps of 
the same degree.
Thus, $d\!\equiv\!m^-\mod2$.
Since 
$$\chi(\Si)=2\big(\chi(\Si_1^+)+\ldots+\chi(\Si_k^+)\big)
=4\big(1\!-\!g(\Si_1^+)+\ldots+1\!-\!g(\Si_k^+)\big)
-2m^-+4m^+,$$
the genus of $\Si$ is odd if $m^-$ and $d$ are even;
the genus of $\Si$ is even if $m^-$ and $d$ are odd.
This establishes the claim for $g_0$ odd.
\end{proof}

\begin{eg}\label{HurCov_eg2}
We now describe real Hurwitz covers of~$\P^1$.
Let
$$\tau\!\equiv\!\tau_1,\eta\!\equiv\!\eta_1\!:\,\P^1\lra\P^1, 
\qquad \tau(z)=\bar{z},\quad \eta(z)=-1/\bar{z}\,,$$
be the standard representatives of the two equivalence classes of anti-holomorphic
involutions on~$\P^1$.
Denote by $\oD\!\subset\!D^2$ the interior of~$D^2$ and by $\bar\R\!\subset\!\P^1$
the closure of~$\R\!\subset\!\C$.
For each $m\!\in\!\Z^{\ge0}$, set
\begin{equation*}\begin{split}
U_m(\eta)&=\big\{(z_1,\ldots,z_m)\!\in\!(\P^1)^m\!:
\,z_i\!\neq\!z_j,\eta(z_j)~\forall\,i\!\neq\!j\big\},\\
U_m&=\big\{(z_1,\ldots,z_m)\!\in\!(\oD)^m\!:\,z_i\!\neq\!z_j~\forall\,i\!\neq\!j\big\},
\end{split}\end{equation*}
and let
$$U_m^{\R}\subset \big\{(x_1,\ldots,x_{2m})\!\in\!\bar\R^{2m}\!:
\,x_i\!\neq\!x_j~\forall\,i\!\neq\!j\big\}$$
so that $x_2$ follows $x_1$, $x_3$ follows $x_2$, etc., with respect to the positive 
direction on~$\bar\R$.
The $m$-th symmetric group~$\bS_m$ acts freely on~$U_m$ by interchanging the coordinates.
The $m$-th cyclic group 
$\Z_m$ acts on~$U_m^{\R}$ by cyclically permuting pairs of consequence coordinates,
i.e.
$$(x_1,\ldots,x_{2m}) \lra (x_3,x_4,\ldots,x_{2m-1},x_{2m},x_1,x_2).$$
Let $\bS_m'$ be the group of automorphisms of~$U_m(\eta)$ generated by 
the interchanges of the coordinates and the maps $z_i\!\lra\!-1/\bar{z}_i$
on each coordinate separately
and $\fU_m(\eta)$ be the orientation double of~$U_m(\eta)/\bS_m'$.
If $(\Si,\si)$ is a genus~$g$ symmetric surface,
\begin{alignat}{1}
\label{tau2cov_e}
\fM(\P^1,2;\fJ_0)^{\tau,\si} &\approx\begin{cases} 
U_{g+1}/\bS_{g+1}\sqcup U_1^{\R}\!\times\!(U_g/\bS_g),&\hbox{if}~
|\pi_0(\Si^{\si})|=1,\\
(U_k^{\R}/\Z_k)\!\times\!(U_{g+1-k}/\bS_{g+1-k}),&
\hbox{if}~k=|\pi_0(\Si^{\si})|\neq1;
\end{cases}\\
\label{eta2cov_e}
\fM(\P^1,2;\fJ_0)^{\eta,\si}&\approx\begin{cases} 
\fU_{g+1}(\eta),&\hbox{if}~\Si^{\si}\!=\!\eset,~g\!\not\in\!2\Z,\\
\eset,&\hbox{otherwise}.
\end{cases}
\end{alignat}
In particular, all these moduli spaces are orientable.
\end{eg}

\begin{proof}
By \cite[p254]{GH}, every degree~2 Hurwitz cover $\Si\!\lra\!\P^1$ not branched over 
$\i\!\in\!\P^1$ can be written~as
$$u\!:\Si'=\big\{(z,w)\!\in\!\C^2:\,w^2\!=\!(z\!-\!z_1)\!\ldots\!(z\!-\!z_{2g+2})\big\}
\lra\C, \qquad (z,w)\lra z,$$
where $z_1,\ldots,z_{2g+2}\!\in\!\C$ are the distinct branched points;
$\Si'$ is completed to the closed surface~$\Si$ by gluing two copies of 
a small disk around $\i\!\in\!\P^1$.\\

\noindent
(1) If $u\!:(\Si,\si,\fJ)\!\lra\!(\P^1,\tau,\fJ_0)$ is a Hurwitz double cover, 
the set of branched points 
\hbox{$z_1,\ldots,z_{2g+2}\!\in\!\C$} is preserved by the involution~$\tau$.
Thus, we can assume~that
$$(z_1,\ldots,z_{2k})\in U_k^{\R} \qquad\hbox{and}\qquad
(z_{2k+1},z_{2k+3},\ldots,z_{2g+1})\in U_{g+1-k}$$
for some $k\!=\!0,1,\ldots,g\!+\!1$.
There are two lifts of~$\tau$ to an involution on~$\Si$:
$$\tau_{\pm}\!: \Si\lra\Si, \qquad (z,w)\lra(\bar{z},\pm\bar{w}).$$
If $k\!\in\!\Z^+$, the fixed locus of the involution~$\tau_+$ consists of~$k$ 
circles containing $\{z_{2k},z_1\}$ and $\{z_{2i},z_{2i+1}\}$ with $i\!=\!1,\ldots,k\!-\!1$,
while the fixed locus of the involution~$\tau_-$ consists of~$k$ 
circles containing $\{z_{2i-1},z_{2i+1}\}$ with $i\!=\!1,\ldots,k$.
If $k\!=\!0$, $\Si^{\tau_+}$ is a single circle, while $\Si^{\tau_-}\!=\!\eset$.
This establishes~\eref{tau2cov_e}.
Since the actions of $\bS_m$ on~$U_m$ and of $\Z_m$ on~$U_m^{\R}$ 
are free and orientation-preserving, the quotients in~\eref{tau2cov_e} are orientable.\\

\noindent
(2) If $u\!:(\Si,\si,\fJ)\!\lra\!(\P^1,\eta,\fJ_0)$ is a Hurwitz double cover, 
the set of branched points $z_1,\ldots,z_{2g+2}\!\in\!\C$ is preserved by 
the involution~$\tau$.
Thus, we can assume that 
$$(z_1,z_2,\ldots,z_{g+1})\in U_{g+1}(\eta)\,, 
\qquad z_{g+1+i}=-1/\bar{z}_i~~\forall\,i=1,\ldots,g\!+\!1\,.$$
There are precisely two lifts of~$\eta$ to an automorphism on~$\Si$:
\BE{etapm_e}\eta_{\pm}\!: \Si\lra\Si, \qquad 
(z,w)\lra\big(-1/\bar{z},(z_1\ldots z_{2g+2})^{1/2}\bar{w}/\bar{z}^{g+1}\big).\EE
These automorphisms are of order 4 if $g$ is even and are involutions if $g$ is odd.
This establishes the second case in~\eref{eta2cov_e} and shows that 
$\fM(\P^1,2;\fJ_0)^{\eta,\si}$ is some double cover of $U_{g+1}(\eta)/\bS_{g+1}'$
in the first case  in~\eref{eta2cov_e}.\\

\noindent
The automorphisms of~$U_{g+1}(\eta)$ interchanging the coordinates are orientation-preserving,
while those conjugating them are orientation-reversing.
Thus, $w_1$ of $U_{g+1}(\eta)/\bS_{g+1}'$ is supported on the loops 
generated by paths in~$U_{g+1}(\eta)$ from $z_i$ to $-1/\bar{z}_i$, such~as 
$$\bI\lra U_{g+1}(\eta), \qquad
t\lra (z_1,\ldots,z_{i-1},(1\!-\!t+t/|z_i|^2)\ne^{\pi\fI t}z_i,z_{i+1},\ldots,z_{g+1}\big).$$
The involution~\eref{etapm_e} along this loop is given~by
$$(z,w)\lra\big(-1/\bar{z},\ne^{\pi\fI t}(z_1\ldots z_{2g+2})^{1/2}\bar{w}/\bar{z}^{g+1}\big),
\qquad t\in\bI,$$
i.e.~this loop lifts to a non-closed path in $\fM(\P^1,2;\fJ_0)^{\eta,\si}$.
Thus, $\fM(\P^1,2;\fJ_0)^{\eta,\si}$ is orientable, and so is 
the orientation double cover of~$U_{2g+2}/\bS_{2g+2}'$.
\end{proof}

\begin{thm}\label{HurCov_thm}
Let $(\Si_0,\si_0)$ and $(\Si,\si)$ be symmetric surfaces and 
$\fJ_0$ be a complex structure on~$\Si_0$ such that $\si_0^*\fJ_0\!=\!-\fJ_0$.
The moduli space $\fM(\Si_0,d;\fJ_0)^{\si_0,\si}$ of degree~$d$ real  Hurwitz covers
$$(\Si,\si,\fJ)\lra(\Si_0,\si_0,\fJ_0)$$ 
is orientable~if
either $(\Si_0,\si_0)\!=\!(\P^1,\tau)$, or $\Si_0\!=\!\bT$, or  
$(\Si_0,\si_0)\!=\!(\P^1,\eta)$ and $d\!\le\!2$.
\end{thm}

\begin{proof}
(1) Suppose $(\Si_0,\si_0)\!=\!(\P^1,\tau)$.
The involution $\tau$ naturally lifts to an involution~$\ti\tau$ on $\cO_{\P^1}(1)$,
so that the usual Euler sequence for~$\P^1$ becomes a short exact sequence
of real bundle pairs: 
$$0\lra (\P^1\!\times\!\C,\tau\!\times\!\id_{\C})\lra
2\big(\cO_{\P^1}(1),\ti\tau\big)\lra (T\P^1,\tnd\tau)\lra0,$$
where $\fc_{\C}$ is the standard conjugation on~$\C$.
Thus,
$$\La_{\C}^{\top}(T\P^1,\tnd\tau)\approx \big(\cO_{\P^1}(1),\ti\tau\big)^{\otimes2}.$$
Since $(\P^1)^{\tau}\!=\!S^1$ is one-dimensional, $(\P^1)^{\tau}$ is orientable and
$$w_2(T(\P^1)^{\tau})=0=w_1\big(\cO_{\P^1}(1)^{\ti\tau}\big)^2\in H^2((\P^1)^{\tau};\Z_2)=\{0\}.$$
Thus, the conclusion in this case follows from \cite[Theorems~1.1,~1.2]{GZ2}
and Theorem~\ref{thm_maps}.\\

\noindent
(2) Suppose $\Si_0\!=\!\bT\!=\!\R^2/\Z^2$.
Since any complex structure on~$\Si_0$ is Kahler, 
the line bundle $(T\bT,\tnd\si_0)$ admits a real square root;
see \cite[Lemma~1.5]{Teh}.
In this case, this can be explicitly seen as follows.
There are three equivalence classes of orientation-reversing involutions on~$\bT$:
$$\ft_0,\ft_1,\ft_2\!:\bT\lra \bT, \qquad
\ft_0(u,v)=(u,-v),\quad \ft_1(u,v)=(v,u),\quad \ft_2(u,v)=(u+\frac12,-v);$$
see \cite[Section~9]{AG}, for example.
In all three cases, the tangent bundle is trivial as a real bundle pair:
\begin{alignat*}{2}
(T\bT,\tnd\ft_0)&\lra(\bT\!\times\!\C,\ft_0\!\times\!\fc_{\C}), &\qquad
(u,v,u',v')&\lra(u,v,u'+\fI v'),\\
(T\bT,\tnd\ft_1)&\lra(\bT\!\times\!\C,\ft_1\!\times\!\fc_{\C}), &\qquad
(u,v,u',v')&\lra (u,v,(u'\!+\!v')+\fI(u'\!-\!v')),\\
(T\bT,\tnd\ft_2)&\lra(\bT\!\times\!\C,\ft_2\!\times\!\fc_{\C}),&\qquad
(u,v,u',v')&\lra (u,v,u'+\fI v').
\end{alignat*}
In particular,
$$(T\bT,\tnd\ft_k)\approx (\bT\!\times\!\C,\ft_k\!\times\!\fc_{\C})^{\otimes2}
\qquad\forall\,k\!=\!0,1,2.$$
Since $\bT^{\ft_k}$ is one-dimensional (and consists of $2\!-\!k$ circles), 
$\bT^{\ft_k}$ is orientable and
$$w_2(T\bT^{\ft_k})=0=w_1\big((\bT\!\times\!\C)^{\ft_k\times\fc_{\C}}\big)^2
\in H^2(\bT^{\ft_k};\Z_2)=\{0\}.$$
Thus, the conclusion in this case also follows from \cite[Theorems~1.1,~1.2]{GZ2}
and Theorem~\ref{thm_maps}.\\

\noindent
(3) A degree 1 map $(\Si,\si,\fJ)\!\lra\!(\P^1,\eta,\fJ_0)$ is an isomorphism,
and so it is sufficient to assume that $(\Si,\si)\!=\!(\P^1,\eta)$ in 
the degree~1 case.
By the explicit description in \cite[Appendix~A.1]{Teh},
$$\fM(\P^1,1;\fJ_0)^{\eta,\eta}\approx\R\P^3\,.$$
The $d\!=\!2$ case for $(\Si_0,\si_0)\!=\!(\P^1,\eta)$ is addressed 
by Example~\ref{HurCov_eg2}.
\end{proof}

\begin{rmk}\label{HurCov_rmk}
The conclusions of Theorem~\ref{HurCov_thm} for $\Si\!=\!\P^1$ and
$\si_0,\si\!=\!\tau,\eta$, without any degree restrictions, 
are implied by \cite[Theorem~1.1]{GZ2} and
are obtained in  \cite[Appendix~A.1]{Teh} by explicitly describing 
$\fM(\P^1,d;\fJ_0)^{\si_0,\si}$.
At this point, we are unaware of any non-orientable moduli spaces 
$\fM(\Si,d;\fJ_0)^{\si_0,\si}$.
It would be interesting to know which of the spaces $\fM(\Si,d;\fJ_0)^{\si_0,\si}$
are orientable (if not all of them are) and which of them are empty
and to obtain analogues of Theorem~\ref{HurCov_thm} and Example~\ref{HurCov_eg1},
respectively, in the most general situation.
This appears to be a purely combinatorial problem about Hurwitz covers.
\end{rmk}

\appendix

\section{Extensions of Theorem~\ref{thm_maps}}
\label{extens_app}

\noindent
In this appendix, we describe an extension of Theorem~\ref{thm_maps}; see
Theorem~\ref{thm_maps2} below.
We make use of what can be seen as an alternative formulation of 
\cite[Proposition~1.2]{Remi},   
which appears to have broader applications to the orientability problem
than  Proposition~\ref{bndlsplit_prp0};
see Lemma~\ref{ext_lmm} below.
The proof of Lemma~\ref{ext_lmm} consists of two main parts.
The first involves reduces the relevant sign computation for a vector bundle
isomorphism over an arbitrary diffeomorphism of~$(\Si,\si)$ to 
a sign computation for an isomorphism over the identity on~$\Si$;
the idea behind this step comes entirely from~\cite{Remi}.
The second part of the proof is handled in completely different ways
in~\cite{Remi} and below: 
the argument in~\cite{Remi} relies on a technical computation at the heart
of~\cite{Remi0}, while ours makes use of a more topological sign computation  
in~\cite{GZ2}.\\

\noindent
It is not clear to us at this point how useful the extension described in this 
appendix is.
In particular, it does not enlarge the class of the complete intersections
$X_{n;\a}\!\subset\!\P^n$ to which Corollary~\ref{CIorient_crl} applies for all 
$B\!\in\!H_2(X;\Z)$.
For the classes of the form $B\!=\!2B'$, with $B'\!\in\!H_2(X;\Z)$,
Theorem~\ref{thm_maps2} does extend the conclusion of Corollary~\ref{CIorient_crl}(1) 
to all complete intersections~with $a_1\!+\!\ldots\!+\!a_k\equiv n\mod2$.
However, \cite[Section~2]{Ge2} implies that natural partial compactifications
of these spaces are not generally orientable in the new cases 
of Corollary~\ref{CIorient_crl}(1) provided by Theorem~\ref{thm_maps2}.

\begin{thm}\label{thm_maps2} 
Let $(X,\om)$ be a symplectic $2n$-manifold with an involution~$\phi$ and
$B$, $J$, $l$, and $(\Si,\si)$ be as in the statement of Theorem~\ref{thm_maps}.
If there exist a real bundle pair $(E,\ti\phi_E)\!\lra\!(X,\phi)$ such~that
\BE{realorient_e1}w_2(TX^{\phi})=w_1(E^{\ti\phi_E})^2
\qquad\hbox{and}\qquad 
\frac12\lr{c_1(TX),B}+\lr{c_1(E),B}\in2\Z\EE
and a rank~1 real bundle pair $(L,\ti\phi_L)\!\lra\!(X,\phi)$ such~that 
\BE{realorient_e2}
\La_{\C}^{\top}(TX,\tnd\phi)=(L,\ti\phi_L)^{\otimes 2},\EE
then the two conclusions of Theorem~\ref{thm_maps} still hold.
Furthermore, \eref{realorient_e1} alone suffices if $\Si\!-\!\Si^{\si}$ is disconnected,
while \eref{realorient_e2} alone suffices if $\Si^{\si}\!=\!\eset$.
\end{thm}

\begin{lmm}[{\cite[Proposition~1.2]{Remi}}]\label{ext_lmm}
Let $( \Si,\si)$ be a symmetric surface,
$G\!:( \Si,\si)\!\lra\!( \Si,\si)$ be a real orientation-preserving diffeomorphism,
$(W,\ti\phi)$ be a rank~$n$ real bundle pair over~$(M_G,\si_G)$, and $(M_G^{\si_G})_i$, for $i\!=\!1,\ldots,m$, be the fixed components of $\si_G$.
If 
$c_1(W)|_{ \Si_s}\!=\!0$ for any $s\!\in\!\bI$ and
\BE{extlmm_e0}\sum_{i=1}^m \blr{w_2(W^{\ti\phi}),[(M_G^{\si_G})_i]_{\Z_2}}=0,\EE 
then 
\BE{extlmm_e} w_1(\det\dbar_{(W,\ti\phi)}) = w_1(\det\dbar_{\La_{\C}^{\top}(W,\ti\phi)})+
(n\!-\!1)\,w_1(\det\dbar_{(M_G\!\times\!\C,\si_G\!\times\!\fc_{\C})}).\EE
\end{lmm}

\begin{proof}
By \cite[Propositions~4.1,~4.2]{BHH}, 
$$(W,\ti\phi)=\big(\bI\!\times\! \Si\!\times\!\C^n,
\id_{\bI}\!\times\!\si\!\times\!\fc_{\C^n}\big)/\!\!\sim_{(G,g)},
\quad\hbox{where}~~~
(1,z,v)\sim_{(G,g)}\big(0,G(z),g(z)v\big)\quad\forall\,(z,v)\!\in\! \Si\!\times\!\C^n,$$
for some map $g\!: \Si\!\lra\!\tU(n)$ such that $g(\si(z))\!=\!\ov{g(z)}$.
Let
\begin{alignat*}{2}
\ti{G}\!:\{1\}\!\times\!\Si\!\times\!\C^n&\lra \{0\}\!\times\!\Si\!\times\!\C^n, &\qquad
\ti{G}(1,z,v)&=\big(0,G(z),g(z)v\big),\\
\det\ti{G}\!:\{1\}\!\times\!\Si\!\times\!\C&\lra \{0\}\!\times\!\Si\!\times\!\C, &\qquad
\det\ti{G}(1,z,v)&=\big(0,G(z),(\det g(z))v\big).
\end{alignat*}
Similarly to the proof of \cite[Proposition~1.2]{Remi}, we write
$$\ti{G}=\big\{\det\ti{G}\oplus G\!\times\!\id_{\C^{n-1}}\big\}\circ 
\big\{\det\ti{G}\oplus G\!\times\!\id_{\C^{n-1}}\big\}^{-1}\circ\ti{G}\,.$$
Choose an orientation on  the determinant bundle of $\dbar$ on 
$W_{|\bI\times\Si}$ by a choice of trivializations as in \cite[Section 4.1]{GZ2}. 
The (exponent of the) sign of the isomorphism induced by~$\ti{G}$ between 
the determinant lines of $\dbar_{(W,\ti\phi)}$ over $s\!=\!1$ and $s\!=\!0$ is 
the sum of the the signs induced by the isomorphisms
\BE{compmaps_e}\begin{split}
\det\ti{G}\oplus G\!\times\!\id_{\C^{n-1}}\!:
\{1\}\!\times\!\Si\!\times\!\C^n&\lra \{0\}\!\times\!\Si\!\times\!\C^n
\qquad\hbox{and}\\
\big\{\det\ti{G}\oplus G\!\times\!\id_{\C^{n-1}}\big\}^{-1}\circ\ti{G}\!:
\{1\}\!\times\!\Si\!\times\!\C^n&\lra \{1\}\!\times\!\Si\!\times\!\C^n.
\end{split}\EE
The latter map covers the identity and can be written~as
$$(1,z,v)\lra \big(1,z,h(z)v\big)$$
for some $h:\Si\!\lra\!\SU(n)$ such that $h(\si(z))\!=\!\ov{h(z)}$.
By \cite[Proposition 4.2]{GZ2} applied with 
$$(X,\phi)=\big(S^1\!\times\!\Si,\id_{S^1}\!\times\!\si\big), \qquad 
(V,\ti\phi)=\big(\bI\!\times\!\Si\!\times\!\C^n/\!\!\sim_{(\id_{\Si},h)},
\si\!\times\!\fc_{\C^n}\big)$$
and~\eref{extlmm_e0}, the sign induced by this map equals
$$\sum_{i=1}^m \blr{w_2(W^{\ti\phi}),(M_G^{\si_G})_i}=0\,;$$
the equivariant $w_2$ in \cite[Proposition 4.2]{GZ2} vanishes by \cite[Lemma~2.1]{Teh},
since $h$ takes values in~$\SU(n)$.
The sign induced by the first map in~\eref{compmaps_e} gives~\eref{extlmm_e}.
\end{proof}

\begin{crl}\label{ext_crl}
Let $( \Si,\si)$ be a symmetric surface,
$G\!:( \Si,\si)\!\lra\!( \Si,\si)$ be a real orientation-preserving diffeomorphism,
$(W,\ti\phi)$ be a real bundle pair over~$(M_G,\si_G)$, and $(M_G^{\si_G})_i$, 
for \hbox{$i\!=\!1,\ldots,m$}, be the fixed components of $\si_G$. If
$$\sum_{i=1}^m \blr{w_2(W^{\ti\phi})\!+\!w_1(W^{\ti\phi})^2,[(M_G^{\si_G})_i]_{\Z_2}}=0,$$ 
then 
$$w_1(\det\dbar_{(W,\ti\phi)}) + w_1(\det\dbar_{(\La_{\C}^{\top}(W,\ti\phi))^*})
=(n\!+\!1)\,w_1(\det\dbar_{(M_G\!\times\!\C,\si_G\!\times\!\fc_{\C})}).$$
 \end{crl}

\begin{proof}
Applying Lemma~\ref{ext_lmm}  with  $W\oplus\big(\La_{\C}^{\top}(W,\ti\phi)\big)^*$, 
we obtain the result as in the proof of Proposition~\ref{bndlsplit_prp}.  
\end{proof}

\begin{proof}[{\bf \emph{Proof of Theorem \ref{thm_maps2}}}] 
We follow the proof of Theorem~\ref{thm_maps} with the bundle $TX\oplus 2E$ in 
place of $TX\oplus 2L$. 
The only difference  in the proof is showing that the first Stiefel-Whitney class 
of the  determinant bundle of a $\dbar$-operator on the pull-back of 
$$\La_\C^{\top}( TX\oplus 2E, \text{d}\phi\oplus 2\ti\phi_E)
= (L\otimes \La_{\C}^{\top}E, \ti\phi_L\otimes\La_{\C}^{\top}\ti\phi_E)^{\otimes 2}
\equiv(W,\ti\phi_W)$$ 
equals that on the trivial rank 1 real bundle pair. 
As shown in the proof of \cite[Corollary~3.3]{GZ2},
\begin{equation*}\begin{split}
\sum_{i=1}^m \blr{w_2(W^{\ti\phi})\!+\!w_1(W^{\ti\phi})^2,[(M_G^{\si_G})_i]_{\Z_2}}
&=\sum_{i=1}^m 
\blr{w_1((L\!\otimes\!\La_{\C}^{\top}E)^{\ti\phi_L\otimes\La_{\C}^{\top}\ti\phi_E})^2, 
[(M_G^{\si_G})_i]_{\Z_2}}\\
&= \sum_{i=1}^{m'} 
\blr{w_1((L\otimes\La_{\C}^{\top}E)^{\ti\phi_L\otimes\La_{\C}^{\top}\ti\phi_E}),
[(\Si^{\si})_i]_{\Z_2}},
\end{split}\end{equation*}
where $\Si_i^{\si}$, for $i=1,\ldots,m'$, are the components of $\Si^{\si}$.
By \cite[Propositions 4.1,4.2]{BHH}, the last expression is 0 if $2|c_1(L\!\otimes\!E)$.
Corollary \ref{ext_crl} now completes the proof.
\end{proof}

\vspace{.2in}

\noindent
{\it Department of Mathematics, Princeton University, Princeton, NJ 08544\\
pgeorgie@math.princeton.edu}\\

\noindent
{\it Department of Mathematics, SUNY Stony Brook, Stony Brook, NY 11790\\
azinger@math.sunysb.edu}\\

\end{document}